\documentclass{article}

\usepackage{latexsym}
\usepackage[a4paper]{geometry}
\usepackage{amscd}
\usepackage{graphics}
\usepackage{amsmath}
\usepackage{amssymb}
\usepackage{mathrsfs}
\usepackage{bbm}
\usepackage{amsthm}
\input xy
\xyoption{all}

\newcommand{\N}{\mathbb{N}}                     % the natural numbers
\newcommand{\C}{\mathbb{C}}                     % the complex plane
\newcommand{\T}{\mathbb{T}}                     % the torus
\newcommand{\set}[2]{\left\{{#1}\mid{#2}\right\}}       % the set
\newcommand{\dist}{\mathrm{dist\,}}             % distance
\newcommand{\dom}{\mathrm{dom}\,}		% domain
\newcommand{\Span}[1]{\mathrm{Span}\left(#1\right)}	 % better Span 
\newcommand{\ca}[1]{\mathscr{#1}}			% an abbreviation for mathscr

\newtheorem{mainthm}{\sc Theorem}           % numbered absolutely
\newtheorem{thm}{\sc Theorem}[section]      % numbered within each section
\newtheorem{cor}[thm]{\sc Corollary}        % numbered along with Theorem
\newtheorem{lem}[thm]{\sc Lemma}            % numbered along with Theorem
\newtheorem{rem}[thm]{\sc Remark}	    % numbered along with Theorem  
\author{Marco Abate, Alberto Abbondandolo, Pietro Majer}
\title{Stable manifolds for holomorphic automorphisms}

\date{December 22, 2011}

\begin{document}

\maketitle 

\renewcommand{\theenumi}{\roman{enumi}}
\renewcommand{\labelenumi}{(\theenumi)}

\begin{abstract}
We give a sufficient condition for the abstract basin of attraction of a sequence of holomorphic
self-maps of balls in $\mathbb{C}^d$ to be biholomorphic to~$\mathbb{C}^d$. As a 
consequence, we get a sufficient condition for the
stable manifold of a point in a compact hyperbolic invariant
subset of a complex manifold to be biholomorphic to a complex Euclidean space. Our result
immediately implies previous theorems obtained
by Jonsson-Varolin and by Peters; in particular, 
we prove (without using Oseledec's theory) that the stable manifold of any point where the negative
Lyapunov exponents are well-defined is biholomorphic to a complex Euclidean space. Our approach is based on the solution of a linear control problem 
in spaces of subexponential sequences, and on careful estimates of the norm of the conjugacy operator by a lower triangular matrix on the space of $k$-homogeneous polynomial endomorphisms of~$\C^d$.  
\smallskip

\end{abstract}

\section*{Introduction}

Let $f\colon M\rightarrow M$ be a holomorphic automorphism of a complex manifold and let $\Lambda\subset M$ be a compact hyperbolic invariant subset of $M$, with stable distribution of complex dimension $d$. The stable manifold of each point $x\in \Lambda$, that is the set
\[
W^s(x) = \set{z\in M}{\lim_{n\rightarrow \infty} \dist \bigl(f^n(z),f^n(x)\bigr) = 0},
\]
is the image of a holomorphic injective immersion $W \hookrightarrow M$
of a complex manifold $W$. Such an immersion endowes $W^s(x)$ with the structure of a complex manifold. As a model $W$ one can choose, for instance, the space of sequences
\[
W = \set{(z_n)\subset M}{z_{n+1} = f(z_n) \mbox{ and } \lim_{n\rightarrow \infty} \dist(z_n,f^n(x)) = 0},
\]
which is a $d$-dimensional complex submanifold of the complex Banach manifold consisting of all sequences $(z_n)\subset M$ such that $\dist(z_n,f^n(x))$ is infinitesimal. In this case, the immersion maps each sequence $(z_n)\in W$ into its first element $z_0$. 

The stable manifold $W^s(x)$ is smoothly diffeomorphic to $\C^d$ and it is natural to ask whether it is also biholomorphic to $\C^d$ (such a question was raised for instance by E.\ Bedford in \cite{bed00}). 

The answer turns out to be affirmative when the invariant set $\Lambda$ is a hyperbolic fixed point, as proven by J.-P. Rosay and W.\ Rudin in \cite{rr88}. More generally, M.\ Jonsson and D.\ Varolin proved that the answer is affirmative for almost every point in $\Lambda$, with respect to any invariant probability measure supported on $\Lambda$ (see \cite{jv02}). In the general case, J.\ E.\ Forn{\ae}ss and B.\ Stens{\o}nes have proven that $W^s(x)$ is always biholomorphic to a domain in $\C^d$ (see \cite{fs04}, Remark \ref{domain} in the appendix below also explains why this fact holds), but the question whether it is actually biholomorphic to $\C^d$ remains open.

By the local stable manifold theorem, a neighborhood of each point $f^n(x)$ in $W^s(f^n(x)) = f^n(W^s(x))$ is biholomorphic to the unit ball $B$ about $0$ in $\C^d$, and by reading the maps $f|_{W^s(f^n(x))}$ by means of these parameterizations (suitably chosen), one obtains a sequence of holomorphic maps
\[
f_n \colon  B \rightarrow B
\]
which fix $0$ and are such that for every $n\in \N$, 
\begin{equation}
\label{uniforme}
\nu |z| \leq |f_n(z)| \leq \lambda |z|\; ,\qquad \forall z\in B\;, 
\end{equation}
for some $0 < \nu \leq \lambda<1$. 
For any sequence of holomorphic maps $f=(f_n \colon  B \rightarrow B)_{n\in \N}$ with the above properties, one can define the {\em abstract basin of attraction of $0$} as the set $Wf$ of all sequences $(z_n)_{n\geq m}$ such that $z_{n+1}=f_n(z_n)$ for every $n\geq m$, under the identification of sequences which eventually coincide. Such an object carries a natural complex structure. This construction is due to J.\ E.\ Forn{\ae}ss and B.\ Stens{\o}nes \cite{fs04}; see Section \ref{saba} below for a more categorical approach.
When the maps $f_n$ are induced by a diffeomorphism $f$ as above, the manifold $Wf$ is naturally biholomorphic to the stable manifold of $x$. It has been conjectured that under the assumptions (\ref{uniforme}), the abstract basin of attraction of $0$ with respect to $(f_n)$ is biholomorphic to $\C^d$ (see \cite{pet05}, \cite{pw05}, \cite{pet07}). A positive answer to this conjecture would imply that the stable manifold of any orbit on a compact hyperbolic invariant set is biholomorphic to $\C^d$.

Notice that by the application of a suitable non-autonomous linear unitary conjugacy, that is by the replacement of $(f_n)$ by $(U_{n+1} \circ f_n \circ U_n^{-1})$ for a sequence $(U_n)$ of unitary automorphisms of $\C^d$, one may always assume that the linear parts of $f_n$,
\[
L_n := Df_n(0)\;,
\]     
are lower triangular matrices. More precisely, the choice of $U_0$ in the unitary group determines the subsequent matrices $U_n$, $n\geq 1$, up to a conjugacy with a unitary diagonal matrix, and uniquely determines the absolute values of the diagonal entries of $L_n$. The aim of this paper is to prove that, under a suitable assumption on the diagonal entries of $L_n$, the above conjecture holds true:

\begin{mainthm}
Let $f_n\colon  B \rightarrow B$ be a sequence of holomorphic maps which satisfies (\ref{uniforme}). Denote by $\lambda_n(j)$ the $j$-th diagonal entry of the lower triangular matrix $L_n=Df_n(0)$, for $1\leq j \leq d$, and assume that for every $\theta>1$ there exists a number $C(\theta)$ such that
\begin{equation}
\label{Ord}
\max_{1\leq h \leq d} \prod_{k=n}^{n+\ell-1} |\lambda_k (h)| \max_{1\leq i \leq j \leq d} \prod_{k=n}^{n+\ell-1} \frac{|\lambda_k(j)|}{|\lambda_k(i)|} \leq C(\theta) \theta^n \lambda^{\ell}\; ,
 \end{equation}
 for every $n,\ell\in \N$, for some $\lambda<1$. Then the abstract basin of attraction of $0$ with respect to $(f_n)$ is biholomorphic to $\C^d$.
 \end{mainthm}
 
Condition (\ref{Ord}) is a sort of asymptotic weak monotonicity requirement on the diagonal entries of $L_n$, and (as it will be made clear by the basic estimate of Lemma~\ref{quoz}) it is very natural in this context.  It is automatically fulfilled when the constants which appear in (\ref{uniforme}) satisfy $\nu^{-1}\lambda^2<1$ (up to the choice of a larger $\lambda<1$, see Remark \ref{lin} below), which is another condition often appearing in the literature.

The above theorem sharpens a result of H. Peters (see \cite[Theorem 9]{pet07}), where a stronger pointwise condition is considered: it is required that there exists $0<\lambda<1$ such that
\begin{equation}
\label{peters}
\max_{1\le h\le d}|\lambda_n(h)|\max_{1\le i\le j\le d}\frac{|\lambda_n(j)|}{|\lambda_n(i)|}\le\lambda
\end{equation}
for all $n$.  See also \cite{for04}, \cite{wol05}, \cite{pw05} and \cite{pvw08} for related results.

An advantage of the asymptotic condition (\ref{Ord}) as opposed to the pointwise condition (\ref{peters}) is that it better fits with ergodic theory. Indeed, 
Theorem 1 implies that the stable manifold $W^s(x)$ of a point in the compact hyperbolic invariant set $\Lambda$ is biholomorphic to $\C^d$ when the negative Lyapunov exponents of $f$ at $x$ are well-defined:
 
\begin{mainthm}
Let $f \colon  M \rightarrow M$ be a holomorphic automorphism of a complex manifold and let $\Lambda$ be a compact hyperbolic invariant set with $d$-dimensional stable bundle $E^s$. Let $x\in \Lambda$ and assume that the
stable space $E^s(x)$ at $x$ has a splitting
\[
E^s(x) = \bigoplus_{i=1}^r E_i 
\]
such that
\[
\lim_{n\rightarrow \infty} |Df^n(x)u|^{1/n} = \bar{\lambda}_i \quad \mbox{uniformly for $u$ in the unit sphere of } E_i \; ,
\]
where the numbers $0<\bar\lambda_i<1$ are pairwise distinct. Then the stable manifold of $x$ is biholomorphic to $\C^d$.
\end{mainthm}

By Oseledec's theorem, the hypotheses in this statement hold for almost every point $x$, with respect to any invariant probability measure on $\Lambda$, and so this theorem is a somewhat more precise statement than the previously mentioned result of M. Jonsson and D. Varolin; thus using our approach it is possible to recover in a unified way all the main results on this subject 
already present in the literature.

Our proof of the above theorems is based on two steps, that we keep separate. The first and main step is a formal non-autonomous conjugacy result, which states that under the assumption (\ref{Ord}) there exists a sequence of formal series $(h_n)$ and a sequence $(g_n)$ of special triangular automorphisms of $\C^d$ (see Section \ref{scoh} for the definition) such that
\[
h_{n+1} \circ f_n = g_n \circ h_n\; , \qquad \forall n\in \N\; .
\]
Here the important fact is that both $(g_n)$ and $(h_n)$ can be chosen to have subexponential growth (roughly speaking, subexponential estimates here replace the slowly varying functions of \cite{jv02}); see Section \ref{sfac} for a precise statement. The main point in the proof of this first step is a careful estimate, proven in Section \ref{scoh}, on the norm of the conjugacy operator by a lower triangular matrix on the space of $k$-homogeneous polynomial endomorphisms of $\C^d$. It is this extimate that leads to condition (\ref{Ord}).

We also exhibit a counterexample which shows that the first step fails when condition (\ref{Ord}) is not fulfilled. Indeed, we show that if the strictly increasing sequence of natural numbers $(s_k)$ grows fast enough, for instance if $s_{k+1}=10^{s_k}$, then the sequence of automorphisms of $\C^2$
\[
f_n (z_1,z_2) := \left\{ \begin{array}{ll} 
\Bigl( \frac{1}{4} z_1 - \frac{1}{4} z_2^2, \frac{1}{2} z_2 \Bigr) \;, & \mbox{if } s_{2k} \leq n < s_{2k+1}\; , \\ \Bigl( \frac{1}{2} z_1, \frac{1}{4} z_2 - \frac{1}{4} z_1^2 \Bigr)\;, & \mbox{if } s_{2k+1}\leq n < s_{2k+2} \; , \end{array} \right.
\]
which does not satisfy (\ref{Ord}),
does not admit a \emph{bounded} formal non-autonomous conjugacy at the level of 2-jets to any sequence of special triangular automorphisms. Notice that here (\ref{uniforme}) holds with $\lambda=1/2$ and any $\nu<1/4$ (up to the restriction to a sufficiently small ball), so this example also shows that the above mentioned condition $\nu^{-1}\lambda^2<1$ is sharp for the issue of the existence of a bounded conjugacy to a sequence of special triangular automorphisms.

The second step is the non-autonomous version of the well-known fact that two germs of holomorphic contractions which are conjugated as jets of a sufficiently high degree are actually conjugated as germs. A result of this kind has been proven by F.\ Berteloot, C.\ Dupont and L.\ Molino in \cite{bdm08}. In the appendix of this paper we provide a different and more concise proof, based on the implicit mapping theorem.

\medskip

The authors would like to thank Eric Bedford and Jasmin Raissy for several useful conversations. The support of the INdAM grant \emph{Local discrete dynamics in one, several and infinitely many variables} during the initial stages of this work is gratefully acknowledged. Part of this work was done while the last two authors were visiting the University of Leipzig and the Max Planck Institut f\"ur Mathematik in den Naturwissenschaften. We would like to thank both institutions for their hospitality and the Humboldt Foundation and the Ateneo Italo Tedesco for financial support in the form of a Humboldt Fellowship and of a Vigoni Project.   

\numberwithin{equation}{section}
\section{A lemma in discrete linear control theory}

Let $E$ be a finite dimensional complex vector space, endowed with the norm $|\cdot|$. We denote by $\|\cdot\|_{L(E)}$ the corresponding operator norm on the space of linear endomorphisms of~$E$. 

If $(A_n)$ is a sequence of linear endomorphisms of $E$ and $n\geq m \geq 0$, we denote by $A_{n,m}$ the composition
\[
A_{n,m} = A_{n-1} A_{n-2} \cdots A_m\, , \qquad A_{n,n} = I \; ,
\]
which satisfies, for any $n\geq m \geq \ell \geq 0$, 
\begin{equation}
\label{ciclo}
A_{n,m} A_{m,\ell} = A_{n,\ell}\; .
\end{equation}
With this notation, the general solution $(u_n)$ of the equation
\begin{equation}
\label{due}
u_{n+1} = A_n u_n + b_n\;, \qquad \forall n\in \N\; ,
\end{equation}
can be written in the compact form
\begin{equation}
\label{solution}
u_n = A_{n,0} u_0 + \sum_{j=0}^{n-1} A_{n,j+1} b_j\; ,
\end{equation}
as shown by a direct computation. A sequence $(u_n)\subset E$ is said to be \emph{subexponential} if for every $\theta>1$ there exists $B=B(\theta)>0$ such that
\[
|u_n| \le B\theta^n  
\]
for all $n\ge 0$.

\begin{lem}
\label{Uno}
Let $(A_n)$ be a sequence of linear automorphisms of $E$ such that for every $\theta>1$ there exist positive numbers $C(\theta)$ and $\alpha(\theta)<1$ for which 
\begin{equation}
\label{uno}
\bigl\| A_{n+\ell,n}^{-1} \bigr\|_{L(E)} \leq C(\theta) \theta^n \alpha(\theta)^\ell\; , \qquad \forall n,\ell \in \N\;.
\end{equation}
Then for every subexponential sequence $(b_n)\subset E$, the equation (\ref{due}) has a unique subexponential solution $(u_n)\subset E$. 
\end{lem}

\begin{proof}
Let $\theta>1$. Choose $1<\omega\leq \theta^{1/2}$ such that $\omega\alpha(\theta^{1/2})<1$. Since $(b_n)$ is subexponential, there exists a number $B>0$ such that
\[
|b_n| \leq B \omega^n\;, \qquad \forall n\in \N\; .
\]
Together with the assumption (\ref{uno}) this implies the estimate
\begin{equation}
\label{ssttmmaa}
\begin{split}
|A_{n+\ell,n}^{-1} b_{n+\ell-1}| \leq C(\theta^{1/2}) \theta^{n/2} \alpha(\theta^{1/2})^\ell B \omega^{n+\ell-1} &= C(\theta^{1/2}) B \omega^{-1} \theta^{n/2} \omega^n \left( \alpha(\theta^{1/2}) \omega \right)^\ell \\ &\leq  C(\theta^{1/2}) B  \theta^n \left( \alpha(\theta^{1/2}) \omega \right)^\ell.
\end{split}
\end{equation}
Since $\alpha(\theta^{1/2})\omega<1$, the above estimate implies that for every $n\in \N$ the series
\[
u_n := - \sum_{\ell=1}^{\infty} A_{n+\ell,n}^{-1} b_{n+\ell-1}\;,
\]
which corresponds to formula (\ref{solution}) with
\[
u_0 = - \sum_{\ell=1}^{\infty} A_{\ell,0}^{-1} b_{\ell-1}\;,
\]
converges absolutely. In particular, $(u_n)$ is a solution of  (\ref{due}) and by (\ref{ssttmmaa})
\[
|u_n| \leq \frac{C(\theta^{1/2}) B}{1-\alpha(\theta^{1/2}) \omega} \theta^n.
\]
Since $\theta>1$ is arbitrary, $(u_n)$ is subexponential. Finally, the uniqueness statement
holds because the homogeneous equation $v_{n+1} = A_n v_n$ does not
have non-zero subexponential solutions since, by (\ref{uno}),
\[
|v_n| = | A_{n,0}  v_0| \geq \bigl\|A_{n,0}^{-1} \bigr\|_{L(E)}^{-1} |v_0| \geq \frac{|v_0|}{C(2)} \alpha(2)^{-n} 
\]
diverges exponentially if $|v_0|\neq 0$, because $\alpha(2)<1$.
\end{proof}

\begin{rem}
In general, an equation of the form $u_{n+1} = f_n(u_n) 
= A_n u_n + b_n$ may have no subexponential solution $(u_n)$, even if the 
sequences $(b_n)$ and $(A_n)$ are bounded. Actually, 
by the formula (\ref{solution}), every solution $(u_n)$ satisfies the estimate
\begin{equation}
\label{ub}
|u_n| \leq a^n |u_0| + b \frac{a^n-1}{a-1}\; ,
\end{equation}
where 
\[
a:= \sup_{n\in \N} \|A_n\|_{L(E)} \;, \quad b:= \sup_{n\in \N} |b_n|\; ,
\]
but the constant $a$ in (\ref{ub}) might be sharp for each solution. A one-dimensional example with $a>1$ is given by 
\[
f_n(u) = \begin{cases} \frac{1}{2} u & \mbox{if } n=0 \mbox{ or } (2k)! \leq n
  < (2k+1)!\;, \\   
  2 u -1 & \mbox{if } (2k+1)! \leq n <
  (2k+2)!\;, \end{cases} 
\]
for every $k\in \N$.
Indeed, in this case $a=2$, $b=1$, so (\ref{ub}) gives us
\[
|u_n| \leq 2^n (|u_0| + 1)\;, \qquad \forall n\in \N\; ,
\]
from which we get
\begin{eqnarray*}
|u_{(2k+1)!}| &=& |f_{(2k+1)!-1} \circ \dots \circ f_{(2k)!} ( u_{(2k)!})|
 = 2^{-(2k+1)! + (2k)!} |u_{(2k)!}| \\ 
 &\leq& 2^{-(2k+1)! + 2 (2k)!}
 (|u_0|+1) = 2^{-(2k-1)(2k)!} (|u_0|+1)= o(1)\;. 
\end{eqnarray*}
Since the map $u\mapsto 2u-1$ is a homothety with center $u=1$ and expanding factor $2$,  
\begin{eqnarray*}
|u_{(2k)!}-1| &=& |f_{(2k)!-1} \circ \dots \circ f_{(2k-1)!} (
 u_{(2k-1)!} )-1| \\ 
 & = & 2^{(2k)! - (2k-1)!} |u_{(2k-1)!} -1| =
 2^{(2k)! - (2k-1)!} (1+o(1))\;.
\end{eqnarray*}
Therefore, the equality $|u_n| = 2^{n + o(n)}$ holds on a subsequence, so the constant $a=2$ is sharp in (\ref{ub}), that is no solution $(u_n)$ is $O(\alpha^n)$ with $\alpha<2$. 
\end{rem}

Let $V$ be a linear subspace of $E$, playing the role of a control space.
If we are allowed to perturb the sequence $(b_n)$ by a sequence in $V$ and if $V$ is preserved by all the automorphisms $A_n$, we find the following generalization of Lemma \ref{Uno}, where the operator norm in (\ref{uno}) is replaced by the operator norm on the quotient $E/V$. 

\begin{lem}
\label{Due}
Let $V$ be a linear subspace of $E$ and let $\pi\colon E \rightarrow E/V$
be the quotient projection. Let $(A_n)$ be a subexponential
sequence of linear automorphisms of $E$ such that $A_n V = V$ for every $n\in \N$, and 
assume that for every $\theta>1$ there exist positive numbers $C(\theta)$ and $\alpha(\theta)<1$ for which 
\begin{equation}
\label{unobis}
\bigl\| A_{n+\ell,n}^{-1} \bigr\|_{L(E/V)} \leq C(\theta) \theta^n \alpha(\theta)^{\ell}\; , \qquad \forall n,\ell \in \N\;.
\end{equation}
Then for every subexponential sequence $(b_n)\subset E$ there is a subexponential sequence $(u_n)\subset E$, unique modulo $V$, and a subexponential sequence $(v_n)\subset V$ such that
\[
u_{n+1} = A_n u_n + b_n + v_n\; .
\]
\end{lem}

\begin{proof}
Since $(b_n)$ is subexponential, so is $(\pi \, b_n$) and by Lemma \ref{Uno} applied to the vector space $E/V$ there is a unique subexponential
sequence $(\xi_n)\subset E/V$ such that $\xi_{n+1} = A_n \xi_n +
\pi \, b_n$. Therefore, there is a sequence $u_n \in \xi_n$ such that
\begin{eqnarray*}
|u_n|=|\xi_n| \qquad \mbox{and} \qquad
v_n:= u_{n+1} - A_n u_n - b_n \in V\;.
\end{eqnarray*}
These sequences $(u_n)$ and $(v_n)$ satisfy all the requirements.
\end{proof}

\begin{rem}
All the results of this section hold, with minor changes, if $E$ is an infinite dimensional complex Banach space. 
\end{rem}

\section{Linear conjugacy operators on the space of homogeneous polynomial maps}
\label{scoh}

We endow $\C^d$ with the norm
\begin{equation}
\label{norm}
|z| := \max_{1\leq j \leq d}  |z_j | \; , \qquad \forall z\in \C^d\; \,
\end{equation}
whose open ball of radius $r$ about $0$, that we denote by $B_r$, is a polydisk. The corresponding operator norm on the space of linear endomorhisms of $\C^d$ is denoted by $\|\cdot\|$.

Let $\mathscr{H}^k$ be the vector space of homogeneous polynomial maps $p\colon \C^d \rightarrow \C^d$ of degree $k$. The space $\mathscr{H}^k$ is naturally isomorphic to the space of  $\C^d$-valued symmetric $k$-linear maps on $\C^d$. We use the same symbol to denote the $k$-linear form and the corresponding $k$-homogeneous polynomial, adopting the notation:
\[
p(z) = p[z]^k = p[\underbrace{z,\dots,z}_{k}] \; , \qquad \forall z\in \C^d \;.
\]
The space $\mathscr{H}^k$ is normed by
\begin{equation}
\label{norma}
\|p\| := \sup_{z\in \C^d\setminus \{0\}} \frac{|p(z)|}{|z|^k} = \|p\|_{\infty,B_1}\;.
\end{equation}
Set
\[
\mathbb{A}_k = \set{\alpha\in \N^d}{|\alpha|=k}\;, \quad 
\mathbb{J} =\{1,\dots,d\}\;,
\]
where $|\alpha|:=\alpha_1 + \dots + \alpha_d$ is the degree of the multi-index
$\alpha$. Then  $\mathscr{H}^k$ is spanned by the basis
\[
z^{\alpha} e_i, \quad (\alpha,i) \in \mathbb{A}_k \times 
\mathbb{J}\;,
\]
where $e_1,\ldots, e_d$ is the canonical basis of $\C^d$, and has dimension
\[
\dim \mathscr{H}^k = d \, \mathrm{card}\,  \mathbb{A}_k = d \binom{k+d-1}{d-1}\; .
\]
An equivalent norm on $\mathscr{H}^k$ is clearly the maximum of the absolute values of the coefficients with respect to this basis and the Cauchy formula implies
\begin{equation}
\label{equivalente}
\max_{(\alpha,i) \in \mathbb{A}_k \times \mathbb{J}} |c_{\alpha,i}| \leq \|p\| \leq \binom{k+d-1}{d-1} \max_{(\alpha,i) \in \mathbb{A}_k\times \mathbb{J}} |c_{\alpha,i}|\;,
\end{equation}
for $p(z) = \sum_{\alpha,i} c_{\alpha,i} z^{\alpha} e_i$.

Consider the flag
\[
(0) = E_0 \subset E_1 \subset E_2 \subset \dots \subset E_{d-1} \subset E_d = \C^d\;,
\]
where $E_j=\Span{e_{d-j+1},\ldots,e_d}$ is the space of vectors $(z_1,\dots,z_d)$ such that $z_i=0$ for every $i\leq d-j$. A holomorphic map $f\colon\C^d\rightarrow \C^d$ is said {\em  triangular} if it preserves this flag, that is $f(E_j) \subseteq E_j$ for every $j=0,\dots,d$. It is said {\em strictly triangular} if $f(E_j) \subseteq E_{j-1}$ for every $j=1,\dots,d$. A holomorphic map $f\colon\C^d \rightarrow \C^d$ is triangular (respectively, strictly triangular) if and and only if $f(0)=0$ and for every $j=1\dots,d$ the $j$-the component of $f$ depends only on the variables $z_1,\dots ,z_j$ (respectively, $z_1,\dots, z_{j-1}$). In particular, a linear endomorphism of $\C^d$ is triangular (respectively, strictly triangular) if and only if the associated matrix is lower triangular (respectively, strictly lower triangular). Notice that the composition of triangular maps is still triangular, and it is strictly triangular if at least one of the maps is so. 
If we set
\[
\mathbb{T}_k := \set{(\alpha,i)\in \mathbb{A}_k \times \mathbb{J}}{
\alpha_j = 0 \; \forall j\geq i}\;, 
  \quad \mathscr{T}^k = \mathrm{span}\, \set{z^{\alpha}
    e_i}{(\alpha,i) \in \mathbb{T}_k}\; ,
\]
then a polynomial map 
\[
h\colon\C^d \rightarrow \C^d, \quad h = \sum_{k=1}^m h^k, \quad h^k \in \mathscr{H}^k\;,
\]
is strictly triangular if and only if each $h^k$ belongs to
$\mathscr{T}^k$.

We are interested in polynomial automorphisms of $\C^d$ that are strictly triangular perturbations of triangular linear maps, that is, maps $g\colon\C^d \rightarrow \C^d$ of the form
\[
g(z) = D z + h(z)\; , \qquad \forall z\in \C^d\;, 
\]
where $D$ is a diagonal linear automorphism and $h$ is a strictly triangular polynomial map. Equivalently, $g$ can be written component-wise as
\begin{equation}
\label{triangular}
\begin{split}
g_1(z) & =  \lambda_1 z_1, \\
g_2(z) & = \lambda_2 z_2 + h_2(z_1), \\
& \vdots  \\
g_d(z) & =  \lambda_d z_d + h_d(z_1,\ldots,z_{d-1}),
\end{split}
\end{equation}
where each polynomial $h_j$ depends only on $z_1,\ldots,z_{j-1}$ and vanishes at the origin, and none of the numbers $\lambda_j$ is zero. The above expression easily implies that $g$ is an automorphism and that its inverse has the same form. 
Throughout this paper, we shall briefly refer to automorphisms of the form (\ref{triangular}) as {\em special triangular automorphisms}. More properties of special triangular
automorphisms are proved in Section \ref{triautosec}. 

A linear automorphism $L$ of $\C^d$ induces a linear conjugacy operator $\mathscr{A}_L\colon\ca{H}^k\to\ca H^k$ by setting
\begin{equation}
\label{conop}
\mathscr{A}_L p := L^{-1} \, p \circ L\; .
\end{equation}
This operator depends controvariantly on $L$, that is
\[
\mathscr{A}_{LM} = \mathscr{A}_M \mathscr{A}_L\; .
\]
If the linear automorphism $L$ is lower triangular, then the
subspace $\mathscr{T}^k$ of strictly triangular $k$-homogeneous polynomial maps is
$\ca{A}_L$-invariant. 
If $D$ is a diagonal linear automorphism, with
\[
D = \left(
\begin{matrix}\delta(1) & & 0 \\ & \ddots & \\ 0 & & \delta(d) \end{matrix}\right)\;,
\]
then the conjugacy operator $\mathscr{A}_D$ 
is diagonal with respect to the standard basis of $\mathscr{H}^k$; indeed
\[
\mathscr{A}_D (z^{\alpha} e_i) =  \delta(1)^{\alpha_1} \dots \delta(i)^{\alpha_i-1}
\dots \delta(d)^{\alpha_d} (z^{\alpha} e_i) \; ,
\]
for all $(\alpha,i) \in \mathbb{A}_k \times \mathbb{J}$.
The induced operator on the quotient $\mathscr{H}^k/\mathscr{T}^k$ (still denoted by $\mathscr{A}_D$) is diagonal with respect to the basis
\[
z^{\alpha}e_i + \mathscr{T}^k\;, \quad (\alpha,i) \in (\mathbb{A}_k
\times \mathbb{J} ) \setminus \mathbb{T}_k\;,
\] 
and the eigenvalue corresponding to the eigenvector $z^{\alpha}e_i + \mathscr{T}^k$ is the number
\[
\delta(1)^{\alpha_1} \dots \delta(i)^{\alpha_i-1}
\dots \delta(d)^{\alpha_d} \; ,
\]
for all $(\alpha,i) \in
(\mathbb{A}_k \times \mathbb{J}) \setminus \mathbb{T}_k$.
Therefore its operator norm can be estimated as follows
\begin{equation}
\label{nonmon0}
\|\mathscr{A}_D\|_{L(\mathscr{H}^k/\mathscr{T}^k)} \leq c(k,d)
\max_{(\alpha,i)\in (\mathbb{A}_k\times \mathbb{J}) \setminus
  \mathbb{T}_k} \left| \delta(1)^{\alpha_1} \dots \delta(i)^{\alpha_i-1}
\dots \delta(d)^{\alpha_d} \right| \; .
  \end{equation}
Here, the presence of the constant $c(k,d)$ is due to the fact that
the norm $\|\cdot\|$ defined by (\ref{norma}) is not a monotone function of the coordinates with respect to the standard basis of $\mathscr{H}^k$. Such a constant would have been 1 if $\mathscr{H}^k$ were endowed with the (monotone) maximum norm of the coefficients with respect to the basis $\{z^{\alpha} e_i\}$, so (\ref{equivalente}) implies that a suitable constant in (\ref{nonmon0}) is 
\[
c(k,d) = \binom{k+ d -1}{d-1}\; .
\]
By the definition of the set $\T_k$, we can reformulate the expression for 
the maximum which appears in (\ref{nonmon0}) and get
\begin{equation}
\label{nonmon}
\|\mathscr{A}_D\|_{L(\mathscr{H}^k/\mathscr{T}^k)} \leq c(k,d)
\left( \max_{1\leq h \leq d} |\delta(h)| \right)^{k-1} \max_{1\leq i \leq j \leq d} \frac{|\delta(j)|}{|\delta(i)|}  \; .
\end{equation}

Let $L_1,\dots,L_{\ell}$ be lower triangular linear automorphisms of $\C^d$, and  let $D_1,\dots,D_{\ell}$ be their diagonal parts. Therefore, $L_n = D_n + N_n$ and $L_n^{-1} = D_n^{-1} + \tilde{N}_n$, where $N_n$ and $\tilde{N}_n$ are strictly lower triangular linear endomorphisms of $\C^d$, for every $n=1,\dots,\ell$. Set $F_n := D_n^{-1} N_n$ and $\tilde{F}_n := D_n \tilde{N}_n$. 
Since $D_n$ is the diagonal part of $L_n$, we have
\[
\|L_n\| \geq \max_{1\leq j \leq d} |L_n e_j| \geq
\max_{1\leq j \leq d} |D_n e_j| = \|D_n\| \; ,
\]
hence
\begin{equation}
\label{mons1}
\|\tilde{F}_n\| = \|D_n L_n^{-1} - I \| \leq
\|D_n\| \|L_n^{-1}\| + 1 \leq
\|L_n\| \|L_n^{-1}\|+ 1\; . 
\end{equation}
Symmetrically, we have
\begin{equation}
\label{mons2}
\|F_n\| \leq \|L_n\|\|L_n^{-1}\|+ 1\; . 
\end{equation}
If we denote by $\mathbbm{2}$ the set $\{0,1\}$,
we have the following useful representation formula:

\begin{lem}
For every $p\in \mathscr{H}^k$ the $k$-homogeneous polynomial map $\mathscr{A}_{L_{\ell}\cdots L_1} p$ equals
\begin{equation}
\label{svil}
\sum_{\alpha,\beta^1,\dots,\beta^k} D_{\ell}  F_{\ell}^{\alpha_{\ell}}  \cdots D_1 F_1^{\alpha_1} \, p \bigl[ D_1^{-1} \tilde{F}_1^{\beta^1_1}  \cdots D_{\ell}^{-1} \tilde{F}_1^{\beta^1_{\ell}}, \dots, D_1^{-1} \tilde{F}_1^{\beta^k_1}  \cdots D_{\ell}^{-1} \tilde{F}_{\ell}^{\beta^k_{\ell}} \bigr]\; ,
\end{equation}
where the sum is taken over all $\alpha,\beta^1,\dots,\beta^k$ in $\mathbbm{2}^{\ell}$ with $|\alpha|<d$, $|\beta^1|<d$, $\dots$, $|\beta^k|<d$. 
\end{lem}

\begin{proof}
If $\gamma\in \mathbbm{2}$ is $0$ then $D_n F_n^{\gamma}=D_n$, while if it is $1$ then $D_n F_n^{\gamma} = N_n$. Therefore
\[
L_{\ell} \cdots L_1 =  (D_{\ell} + N_{\ell})  \cdots (D_1 + N_1) = \sum_{\alpha\in \mathbbm{2}^{\ell}} D_{\ell} F_{\ell}^{\alpha_{\ell}}  \dots D_1 F_1^{\alpha_1}\; .
\]
Actually, all the products containing at least $d$ of the $N_n$'s vanish, so the above sum involves only the multi-indices $\alpha$ with $|\alpha|<d$, so
\begin{equation}
\label{des}
L_{\ell} \cdots L_1 =   \sum_{\substack{\alpha\in \mathbbm{2}^{\ell} \\ |\alpha|<d}} D_{\ell} F_{\ell}^{\alpha_{\ell}}  \cdots D_1 F_1^{\alpha_1}\; .
\end{equation}
Similarly,
\begin{equation}
\label{sin}
L_1^{-1} \cdots L_{\ell}^{-1} = \sum_{\substack{\beta\in \mathbbm{2}^{\ell} \\ |\beta|<d}} D_1^{-1} \tilde{F}_1^{\beta_1}  \cdots D_{\ell}^{-1} \tilde{F}_{\ell}^{\beta_{\ell}}\;.
\end{equation}
Formulas (\ref{des}) and (\ref{sin}) imply the identity
\begin{eqnarray*}
\mathscr{A}_{L_{\ell} \cdots L_1} p = \sum_{\substack{\alpha \in \mathbbm{2}^{\ell}, \\ |\alpha|<d}} D_{\ell}  F_{\ell}^{\alpha_{\ell}}  \cdots D_1 F_1^{\alpha_1} \, p \circ \Bigl(  \sum_{\substack{\beta \in \mathbbm{2}^{\ell}, \\ |\beta|<d}} D_1^{-1} \tilde{F}_1^{\beta_1}  \cdots D_{\ell}^{-1} \tilde{F}_{\ell}^{\beta_{\ell}} \Bigr) \\
=  \sum_{\substack{\alpha \in \mathbbm{2}^{\ell}, \\ |\alpha|<d}} D_{\ell}  F_{\ell}^{\alpha_{\ell}}  \cdots D_1 F_1^{\alpha_1} \, p \Bigl[  \sum_{\substack{\beta^1 \in \mathbbm{2}^{\ell}, \\ |\beta^1|<d}} D_1^{-1} \tilde{F}_1^{\beta^1_1}  \cdots D_{\ell}^{-1} \tilde{F}_{\ell}^{\beta^1_{\ell}}, \dots,  \sum_{\substack{\beta^k \in \mathbbm{2}^{\ell}, \\ |\beta^k|<d}} D_1^{-1} \tilde{F}_1^{\beta^k_1}  \cdots D_{\ell}^{-1} \tilde{F}_{\ell}^{\beta^k_{\ell}} \Bigr].
\end{eqnarray*}
By $k$-linearity the latter expression can be rewritten as (\ref{svil}).
\end{proof}

The following lemma is our main estimate for conjugacy operators induced by triangular automorphisms:

\begin{lem}
\label{quoz}
Let $L_n$, $1\leq n \leq {\ell}$, be lower triangular linear automorphisms of $\C^d$. 
Assume that the vector of diagonal entries of $L_n$  is $(\lambda_n(1),\dots,\lambda_n(d))$. Then 
\[
\left\|\mathscr{A}_{L_{\ell} \dots L_1}\right\|_{L(\mathscr{H}^k/\mathscr{T}^k)} \leq K \,\ell^{N}  \max_{(n_r)} \prod_{r=0}^N  \left(\Bigl(
\max_{1\leq h \leq d} |\lambda_{n_{r+1},n_r}(h)| \Bigr)^{k-1} \max_{1\leq i \leq j \leq d} \frac{|\lambda_{n_{r+1},n_r}(j)|}{|\lambda_{n_{r+1},n_r}(i)|} \right)  \; ,
\]
where
\[
K=K\left(d,k,\max_{1\leq n\leq \ell} \|L_n\|,\max_{1\leq n \leq \ell} \|L_n^{-1}\| \right)\; , \quad N:=(k+1)(d-1)\; ,
\]
and the first maximum is taken over all the partitions $1=n_0\leq n_1 \leq \dots \leq n_N \leq n_{N+1} = \ell$.
\end{lem}

\begin{proof}
Fix some $\alpha,\beta^1,\dots,\beta^k \in \mathbbm{2}^{\ell}$ such that $|\alpha|<d$, $|\beta^1|<d$, \dots, $|\beta^k|<d$. The operator
\begin{equation}
\label{operator}
p \mapsto D_{\ell}  F_{\ell}^{\alpha_{\ell}}  \cdots D_1 F_1^{\alpha_1} \, p \left[ D_1^{-1} \tilde{F}_1^{\beta^1_1}  \cdots D_{\ell}^{-1} \tilde{F}_{\ell}^{\beta^1_{\ell}}, \dots, D_1^{-1} \tilde{F}_1^{\beta^k_1}  \cdots D_{\ell}^{-1} \tilde{F}_{\ell}^{\beta^k_{\ell}} \right]
\end{equation}
appearing in (\ref{svil}) can be seen as the composition of conjugacy operators $\mathscr{A}_{D_n}$, $1\leq n \leq \ell$, alternated with the left-multiplication operators
\[
\mathscr{L}_n \, p := F_n^{\alpha_n} p\; , \quad 1\leq n \leq \ell \; ,
\]
and the right-multiplication operators
\[
\mathscr{R}_n \, p := p \bigl[ \tilde{F}_n^{\beta_n^1}, \dots, \tilde{F}_n^{\beta_n^k} \bigr]\; , \quad 1\leq n \leq \ell\; .
\]
Each of these operators preserves the subspace $\mathscr{T}^k$, and by (\ref{mons1}) and (\ref{mons2}),
\begin{equation}
\label{lr}
\begin{split}
\|\mathscr{L}_n\|_{L(\mathscr{H}^k/\mathscr{T}^k)} &\leq \|F_n^{\alpha_n} \| \leq
\|F_n \|^{\alpha_n} \leq \rho^{\alpha_n}\; ,  \\ 
\|\mathscr{R}_n \|_{L(\mathscr{H}^k/\mathscr{T}^k)} &\leq
\|\tilde{F}_n^{\beta^1_n}\|\dots  \|\tilde{F}_n^{\beta^k_n}\| \leq \|\tilde{F}_n \|^{\beta^1_n} \dots  \|\tilde{F}_n \| ^{\beta^k_n} \leq \rho^{\beta_n^1 + \dots + \beta_n^k}\; , 
\end{split}
\end{equation}
where 
\[
\rho:= \max_{1\leq n \leq {\ell}} \bigl( \|L_n\|  \|L_n^{-1}
\|  + 1\bigr)\; .
\]
Since $|\alpha|<d$, the number of indices $n$ for which $\mathscr{L}_n$ is not the identity is at most $d-1$. Similarly, since $|\beta^j|<d$ for every $j=1,\dots,k$, the number of indices $n$ for which $\mathscr{R}_n$ is not the identity is at most $k(d-1)$. It follows that there are natural numbers
\[
1=n_0 \leq n_1 \leq \dots \leq n_{N+1} = \ell
\]
where $N=(k+1)(d-1)$, such that the operator (\ref{operator}) is the composition of the $N+1$ conjugacy operators
\[
\mathscr{A}_{D_{n_{j+1},n_j}}\; ,\qquad j=0,\dots, N\; ,
\]
with some of the operators $\mathscr{L}_n$ and $\mathscr{R}_n$. Using also (\ref{lr}), we deduce that the norm of the operator (\ref{operator}) on $\mathscr{H}^k/\mathscr{T}^k$ is not larger than the number
\begin{equation}
\label{bn}
\begin{split}
\prod_{j=0}^{N} \|\mathscr{A}_{D_{n_{j+1},n_j}}\|_{L(\mathscr{H}^k/\mathscr{T}^k)} \cdot \prod_{n=1}^{\ell} \rho^{\alpha_n} \cdot  \prod_{n=1}^{\ell} \rho^{\beta_n^1 + \dots + \beta_n^k} \\ = \rho^{|\alpha| + |\beta^1| + \dots + |\beta^k|} \prod_{j=0}^{N} \|\mathscr{A}_{D_{n_{j+1},n_j}}\|_{L(\mathscr{H}^k/\mathscr{T}^k)}\; .
\end{split} \end{equation}
Notice that
\begin{equation}
\label{combi}
\sum_{\substack{\alpha \in \mathbbm{2}^{\ell}, \\ |\alpha|<d}} \sum_{\substack{\beta^1 \in \mathbbm{2}^{\ell}, \\ |\beta^1|<d}} \dots \sum_{\substack{\beta^k \in \mathbbm{2}^{\ell}, \\ |\beta^k|<d}} \rho^{|\alpha| + |\beta^1| + \dots + |\beta^k|}    =  \left( \sum_{\substack{\alpha \in \mathbbm{2}^{\ell} \\ |\alpha|<d}}
\rho^{|\alpha|} \right)^{k+1} 
\end{equation}
and
\[
\sum_{\substack{\alpha \in \mathbbm{2}^{\ell} \\ |\alpha|<d}}
\rho^{|\alpha|} =
\sum_{j=0}^{d-1} \binom{\ell}{j}
    \rho^{j}  \leq \rho^{d-1} \sum_{j=0}^{d-1} \binom{\ell}{j}  \leq
    \rho^{d-1} \sum_{j=0}^{d-1} \frac{\ell^j}{j!}  \leq e \rho^{d-1} \ell^{d-1}\; ,
\]
so the quantity (\ref{combi}) is at most $e^{k+1} \rho^N \ell^N$.
Since the operator $\mathscr{A}_{L_{\ell}\cdots  L_1}$ is the sum of the operators (\ref{operator}) over all multi-indices $\alpha,\beta^1,\dots,\beta^k$ in $\mathbbm{2}^{\ell}$ with weight less than $d$, this bound on (\ref{combi}) and the fact that the norm of (\ref{operator}) is at most (\ref{bn}) imply the estimate 
\[
\|\mathscr{A}_{L_{\ell} \dots L_1}\|_{L(\mathscr{H}^k/\mathscr{T}^k)} \leq e^{k+1} \, \rho^{N} \, \ell^{N} \max_{1=n_0 \leq n_1\leq \dots \leq n_N \leq n_{N+1} = \ell} \; \prod_{j=0}^N \left\| \mathscr{A}_{D_{n_{j+1},n_j}} \right\| \, .
\]
The conclusion now follows from the estimate (\ref{nonmon}). 
\end{proof}

\section{The formal non-autonomous conjugacy}
\label{sfac}

Let $\mathscr{F} \subset \C[[z_1,\dots,z_d]]^d$ be the space of formal series in $d$ variables and with $d$ components, with vanishing zero order term. If $f\in \mathscr{F}$ and $k\in \N$, we denote by $f^k$ the $k$-homogeneous part of $f$. Therefore each $f\in \mathscr{F}$ can be written uniquely as
\begin{equation}
\label{expr}
f = \sum_{k=1}^{\infty} f^k\; , 
\end{equation}
where $f^k\in \mathscr{H}^k$ for every $k\in \N_+$, and any expression of the form (\ref{expr}) defines an element of $\mathscr{F}$. The formal composition of two formal series $h,f\in \mathscr{F}$ is well-defined, and its $k$-th homogeneous part is given by the finite sum
\begin{equation}
\label{composizione}
(h \circ f)^k = \sum_{\substack{j\geq 1\\ \alpha\in \N_+^j \\ |\alpha| = k}} h^j[f^{\alpha_1},\dots,f^{\alpha_j}]\; .
\end{equation}
An element $f\in \mathscr{F}$ is invertible with respect to the formal composition if and only if its first order term $f^1$ is an invertible linear mapping. 

A sequence $(f_n)\subset \mathscr{F}$ is said to be \emph{subexponential} if for every $k\in \N_+$ the sequence $(f^k_n)\subset \mathscr{H}^k$ is subexponential. Clearly, a bounded sequence in $\mathscr{F}$ with its standard structure of topological vector space, that is a sequence $(f_n)$ such that $(f_n^k)$ is bounded for every $k$, is subexponential.
The aim of this section is to prove the following: 

\begin{thm}
\label{formal}
Let $(f_n)$ be a subexponential sequence in $\mathscr{F}$ such that  for every $n\in \N$ the linear endomorphism $L_n := f_n^1$ is invertible, lower triangular, and satisfies the uniform bounds
\begin{equation}
\label{stime}
\|L_{n,m}\|  \leq c \lambda^{n-m} \; , \quad \|L_{n,m}^{-1}\|  \leq c \mu^{n-m}  \;, \qquad \forall n\geq m \geq 0\; ,
\end{equation}
where $c>0$ and $0<\lambda<1< \mu$.
Let $m_0\in \N_+$ be such that
\begin{equation}
\label{lambdamu}
\lambda^{m_0+1} \mu < 1\; .
\end{equation}
We assume that the vector  $(\lambda_n (1),\lambda_n(2),\dots,\lambda_n(d))$ of diagonal entries of $L_n$ satisfies the following condition: For every $\theta>1$ there exists a positive number $C(\theta)$ such that 
\begin{equation}
\label{ord}
\max_{1\leq h \leq d} |\lambda_{n+\ell,n}(h)| \max_{1\leq i \leq j \leq d} \frac{|\lambda_{n+\ell,n}(j)|}{|\lambda_{n+\ell,n}(i)|} \leq C(\theta) \theta^n \lambda^{\ell}\; ,
\end{equation}
for every $n,\ell \in \N$.
Then there exist a subexponential sequence $(g_n)$ of  special triangular automorphisms of $\C^d$ of degree at most $m_0$ and a subexponential sequence $(h_n)$ in $\mathscr{F}$ such that  $g_n^1=L_n$, $h_n^1=I$, and
\begin{equation}
\label{conju}
h_{n+1} \circ f_n = g_n \circ h_n \; , \qquad \forall n\in \N\; .
\end{equation}
\end{thm}

\begin{proof}
By the identity (\ref{composizione}), the conjugacy equation (\ref{conju}) is equivalent to the infinite system of algebraic equations
\begin{equation}
\label{coeff}
h_{n+1}^k[f_n^1]^k + \sum_{\substack{1\leq j <k \\ \alpha\in \N_+^j \\ |\alpha|=k}} h_{n+1}^j[f_n^{\alpha_1},\dots,f_n^{\alpha_j}] = g^1_n[h_n^k] + \sum_{\substack{1< j \leq k \\ \alpha\in \N_+^j \\ |\alpha|=k}}   
g_n^j [h_n^{\alpha_1},\dots,h_n^{\alpha_j}] \;,
\end{equation}
for $k\in \N_+$ and $n\in \N$.
If we set
\begin{equation}
\label{ini}
g_n^1 := f_n^1 = L_n\; , \quad h_n^1 := I\; , \qquad \forall n\in \N \;, 
\end{equation}
the equations (\ref{coeff}) are trivially satisfied when $k=1$, for any $n\in \N$. We have to prove that for every $k\geq 2$ there are subexponential sequences $(h_n^k)$ and $(g_n^k)$ in $\mathscr{H}^k$ solving the equations (\ref{coeff}), with $g_n^k$ in the space of strictly triangular $k$-homogeneous polynomial maps $\mathscr{T}^k$ for $2\leq k \leq m_0$, and $g_n^k=0$ for $k>m_0$.

Let $k\geq 2$. Using (\ref{ini}), right-composing by $L_n^{-1}$, and isolating the terms with $j=1$ and $j=k$ in the two sums, we can rewrite equation (\ref{coeff}) as
\[
h_{n+1}^k = L_n h_n^k \circ L_n^{-1} + g^k_n \circ L_n^{-1} - f_n^k \circ L_n^{-1} +  
\sum_{\substack{1< j < k \\ \alpha\in \N_+^j \\ |\alpha|=k}} \bigl(   
g_n^j [h_n^{\alpha_1},\dots,h_n^{\alpha_j}] -  h_{n+1}^j[f_n^{\alpha_1},\dots,f_n^{\alpha_j}]\bigr)  \circ L_n^{-1}\; .
\]
This equation has the form
\begin{equation}
\label{coeff3}
h_{n+1}^k = \mathscr{A}_{L_n^{-1}} h_n^k + g_n^k \circ L_n^{-1} + b_n^k\; , 
\end{equation}
where $\mathscr{A}_{L_n^{-1}} : \mathscr{H}^k \rightarrow \mathscr{H}^k$ is the conjugacy operator by $L_n^{-1}$ defined in (\ref{conop}),
and $b_n^k$ is of the form
\begin{equation}
\label{forma}
b_n^k = B^k \bigl(f_n^1,\dots,f_n^k;g_n^2,\dots,g_n^{k-1};h_n^1, \dots,h_n^{k-1};h_{n+1}^2,\dots,h_{n+1}^{k-1}\bigr)\;.
\end{equation}
In particular, $b_n^k$ does not depend on the sequences $(g^j_n)_{n\in \N}$ and $(h^j_n)_{n\in\N}$ for $j\geq k$.

Let $m \geq 2$ be an integer. Arguing inductively on $m$, we assume that for every $k< m$ the equations (\ref{coeff}), equivalently the equations (\ref{coeff3}), are satisfied by subexponential sequences $(g_n^k)$ and $(h_n^k)$, such that (\ref{ini}) holds and
\begin{equation}
\label{progi}
g_n^k \in \mathscr{T}^k \mbox{ for } 2\leq k < \min\{m,m_0+1\}\;, \quad g_n^k = 0 \mbox{ for } m_0 < k < m\;, \qquad \forall n\in \N\;.
\end{equation}
We want to prove that the equations 
\begin{equation}
\label{coeffm}
h_{n+1}^m = \mathscr{A}_{L_n^{-1}} h_n^m + g_n^m \circ L_n^{-1} + b_n^m\; , 
\end{equation}
are satisfied by subexponential sequences $(g_n^m)$, $(h_n^m)$, such that 
\begin{equation}
\label{progi2}
g_n^m \in \mathscr{T}^m\;, \quad g_n^m = 0 \mbox{ if } m>m_0\;, 
\quad \forall n\in \N\;.
\end{equation}
Our strategy is to apply Lemma \ref{Uno} when $m>m_0$, when we are forced to take $g_n^m=0$, and to apply Lemma \ref{Due} when $m\leq m_0$, by using the space of strictly triangular homogeneous polynomials $\mathscr{T}^m$ as a control space.  

By the assumption that $(f_n)$ is subexponential and by the inductive hypothesis, (\ref{forma}) shows that the sequence $(b^m_n)$ is subexponential. In order to estimate the operator norm of the linear part of the equation, we distinguish the cases $m>m_0$ and $2\leq m\leq m_0$.

\paragraph{The case $\mathbf{m>m_0}$.} By (\ref{stime}), the norm of the operator $\mathscr{A}_{L_{n+\ell,n}}$ on $\mathscr{H}^m$ is bounded by
\[
\|\mathscr{A}_{L_{n+\ell,n}}\|_{L(\mathscr{H}^m)} \leq \|L_{n+\ell,n}^{-1} \| \| L_{n+\ell,n}\|^m 
\leq c^{m+1} \left( \mu \lambda^m\right)^{\ell}\; , \qquad \forall n,\ell \in \N\; .
\]
Since $m\geq m_0+1$ and $\lambda<1$, by (\ref{lambdamu}) the number $\mu\lambda^m$ is smaller than 1, so the operators $A_n=\mathscr{A}_{L_n^{-1}} = \mathscr{A}_{L_n}^{-1}$ satisfy the assumption (\ref{uno}) of Lemma \ref{Uno}.
This lemma implies that the equation (\ref{coeffm}) with $g^m_n=0$ for every $n\in \N$ has a unique subexponential solution $(h_n^m)$. 

\paragraph{The case $\mathbf{2\leq m \leq m_0}$.}  By Lemma \ref{quoz}, the norm of the operator $\mathscr{A}_{L_{n+\ell,n}}$ on the quotient space $\mathscr{H}^m/\mathscr{T}^m$ is bounded by
\[
\begin{split}
\|\mathcal{A}_{L_{n+\ell,n}}\|_{L(\mathscr{H}^m/\mathscr{T}^m)} &\leq K \,\ell^{N}  \max_{(n_r)} \; \prod_{r=0}^N \left( \Bigl(
\max_{1\leq h \leq d} |\lambda_{n_{r+1},n_r}(h)| \Bigr)^{m-1} \max_{1\leq i \leq j \leq d} \frac{|\lambda_{n_{r+1},n_r}(j)|}{|\lambda_{n_{r+1},n_r}(i)|} \right) \\ 
& \leq K \,\ell^{N}  \max_{(n_r)} \; \prod_{r=0}^N \left( 
\max_{1\leq h \leq d} |\lambda_{n_{r+1},n_r}(h)| \max_{1\leq i \leq j \leq d} \frac{|\lambda_{n_{r+1},n_r}(j)|}{|\lambda_{n_{r+1},n_r}(i)|} \right) \;,
\end{split} 
\]
where $N=(m+1)(d-1)$, $K=K(d,m,c\lambda,c\mu)$, and the first maximum is taken over all the partitions $n=n_0\leq n_1 \leq \dots \leq n_N \leq n_{N+1} = n+\ell-1$. 
In the last inequality we have used that $|\lambda_{n_{r+1},n_r}(h)|\leq \lambda^{n_{r+1}-n_r} \leq 1$, by (\ref{stime}), and $m\geq 2$.
Let $\theta$ be any number greater than $1$ and let $C(\theta)$ and $\alpha(\theta)<1$ be positive numbers such that (\ref{ord}) holds. Plugging (\ref{ord}) into the last inequality, we get
\begin{equation}
\label{sstt}
\begin{split}
\|\mathcal{A}_{L_{n+\ell,n}}\|_{L(\mathscr{H}^m/\mathscr{T}^m)} &\leq K \,\ell^{N}  \max_{(n_r)} \; \prod_{r=0}^N C(\theta) \theta^{n_r} \lambda^{n_{r+1} - n_r} \\ &= K \ell^N C(\theta)^{N+1} \theta^{n+N(n+\ell-1)} \lambda^{\ell} \\ &= K \theta^{-N} \ell^N C(\theta)^{N+1} \theta^{(N+1)n} \left(\theta^N \lambda \right)^{\ell}\; .
\end{split}
\end{equation}
Now let $\omega>1$. Choose $\theta>1$ such that $\theta^{N+1}\leq \omega$ and $\theta^N \lambda<1$. Then for every number $\alpha$ such that $\theta^N \lambda<\alpha<1$ there exists a constant $\tilde{C}$ such that
\[
\ell^N \theta^{(N+1)n} (\theta^N \lambda)^{\ell} \leq \tilde{C} \omega^n \alpha^{\ell}\; , \qquad \forall n,\ell\in \N\; .
\]
Together with the above estimate, (\ref{sstt}) shows that $A_n=\mathscr{A}_{L_n^{-1}}$ satisfies the assumption (\ref{unobis}) of Lemma \ref{Due}, which then implies the existence of a subexponential sequence $(\hat{g}_n)$ in $\mathscr{T}^m$ such that the equation
\[
h_{n+1}^m = \mathscr{A}_{L_n^{-1}} h_n^m + \hat{g}_n + b_n^m
\]
has a subexponential solution $(h_n^m)\subset \mathscr{H}^m$. Then (\ref{coeffm}) is solved by the subexponential sequences $(h_n^m)$ and 
\[
g_n^m := \hat{g}_n \circ L_n \in \mathscr{T}^m\; .
\]
This concludes the proof of Theorem \ref{formal}.
\end{proof}

\begin{rem} (Uniqueness)
\label{unique}
Once the sequence of special triangular automorphisms $(g_n)$ has been fixed, the sequence $(h_n)$ is uniquely determined by its first element $h_0$ thanks to (\ref{conju}). Moreover, the above proof shows that the homogeneous polynomials $h_0^k$ for $k> m_0$ are uniquely determined by the ones of degree at most $m_0$.
\end{rem}

\begin{rem}
\label{lin}
(The linearizable case)
If $\lambda^2 \mu<1$ we can take $m_0=1$ and the condition (\ref{ord}) is automatically fulfilled (up to the choice of a larger $\lambda<1$). In this case, the triangular automorphisms $g_n=L_n$ are linear and the subexponential formal conjugacy $(h_n)$ is unique and actually bounded.
\end{rem}

\begin{rem}
\label{indep}
As explained at the beginning of Section \ref{scoh}, the operator norm which appears in (\ref{stime}) is the one induced by the norm (\ref{norm}) on $\C^d$. However, the condition (\ref{stime}) does not depend on the choice of the norm on $\C^d$, up to the choice of a different constant $c$.
\end{rem} 

\begin{rem}
If we strengthen the assumptions of the above theorem by requiring that $(f_n)$ is bounded and that (\ref{ord}) holds with $\theta=1$, the same arguments imply that $(g_n)$ and $(h_n)$ are bounded. However, the weaker assumption that (\ref{ord}) holds  for every $\theta>1$ is more relevant for our purposes, as it holds generically when the $f_n$'s are induced by the restriction of a diffeomorphism $f$ to the stable manifolds along an orbit on a compact hyperbolic set (see the proof of Corollary \ref{ucurullariu} below).
\end{rem}

\section{A counterexample}

If we drop the assumption (\ref{ord}), the conclusion of Theorem \ref{formal} may fail. Let us exhibit a counterexample in dimension $d=2$. Let $(s_k)$  be a strictly increasing sequences of integers such that
\begin{eqnarray}
\label{ordine1}
8^{s_{2k}} &=& o(s_{2k+1}) \qquad \mbox{for } k\rightarrow \infty\;, \\
\label{ordine2}
s_{2k-1} &=& o(s_{2k}) \quad\qquad \mbox{for } k\rightarrow \infty\; . 
\end{eqnarray}
For instance, we may take $(s_k)$ defined recursively by
\[
s_0 = 1\;, \quad s_{k+1} = 10^{s_k}\;,
\]
that is, in Knuth's up-arrow notation, $s_k =  10\uparrow\uparrow k$.

Consider the sequence $(f_n) \subset \mathscr{F}$ defined by
\[
f_n (z_1,z_2) := \left\{ \begin{array}{ll} 
\Bigl( \frac{1}{4} z_1 - \frac{1}{4} z_2^2, \frac{1}{2} z_2 \Bigr) \;, & \mbox{if } s_{2k} \leq n < s_{2k+1}\; , \\ \Bigl( \frac{1}{2} z_1, \frac{1}{4} z_2 - \frac{1}{4} z_1^2 \Bigr)\;, & \mbox{if } s_{2k+1}\leq n < s_{2k+2} \; , \end{array} \right.
\]
where $k$ varies in $\N$. Then the assumption (\ref{stime}) of Theorem \ref{formal} holds with the sharp constants $\lambda=1/2$, $\mu=4$. In particular, we can take $m_0=2$ in (\ref{lambdamu}), but not $m_0=1$, so we are not in the linearizable case of Remark \ref{lin}. 
On the other hand, the vector $(\lambda_n(1),\lambda_n(2))$ of diagonal entries of $f_n^1$ is either $(1/2,1/4)$ or $(1/4,1/2)$, so
\[
\max_{1\leq h \leq 2} |\lambda_{n+\ell,n}(h)| \max_{1\leq i \leq j \leq 2} \frac{|\lambda_{n+\ell,n}(j)|}{|\lambda_{n+\ell,n}(i)|} = (1/2)^{\ell} 2^{\ell} = 1\; ,
\]
and (\ref{ord}) does not hold.

We shall prove that if $(g_n)$ and $(h_n)$ are sequences in $\mathscr{F}$ such that $g_n^1=f_n^1$, $h_n^1=I$,
\[
h_{n+1} \circ f_n = g_n \circ h_n \quad \mbox{as 2-jets ,} \qquad \forall n\in \N\; , 
\]
and $g_n^2$ belongs to $\mathscr{T}^2$ for every $n\in \N$, then $(h_n^2)$ is not subexponential.     

Since $f_n^1$ is diagonal, the equation (\ref{coeff3}) splits into an equation for each element of the standard basis of $\mathscr{H}^2$. Since $g_n^2$ belongs to $\mathscr{T}^2$, its first component vanishes and the equation for the coefficient of $z_2^2$ in the first component of $h_n$ - call it $u_n$ - is just
\[
u_{n+1} =  \lambda_n(1) \lambda_n(2)^{-2}  u_n - \lambda_n(2)^{-2} a_n \; ,
\]
where $a_n$ is the coefficient of $z_2^2$ in the first component of $f_n$. By the definition of $f_n$, the sequence of complex numbers $u_n$ satisfies the recursive equation
\[
u_{n+1} = \left\{ \begin{array}{ll}  u_n + 1 , & \mbox{if } s_{2k}\leq n < s_{2k+1} \; , \\ 8 u_n , & \mbox{if }  s_{2k+1} \leq n < s_{2k+2} \; . \end{array} \right.
\]
The solutions $(u_n)$ of the above equation are uniquely determined by the choice of the first element $u_0\in \C$. We claim that for every choice of $u_0$, the sequence $(u_n)$ is not subexponential. 

By the bound (\ref{ub}), we have
\[
|u_n| \leq 8^n (|u_0|+1)\; .
\]
Therefore, by (\ref{ordine1}),
\[
|u_{s_{2k+1}}| = |u_{s_{2k}} + s_{2k+1} - s_{2k} | \geq s_{2k+1} - s_{2k} - |u_{s_{2k}}| \geq s_{2k+1} - s_{2k} - 8^{s_{2k}} (|u_0| + 1) = s_{2k+1} \bigl( 1 + o(1) \bigr) 
\]
diverges. Thus, if $k$ is large enough, we find by (\ref{ordine2})
\[
|u_{s_{2k}}| = 8^{s_{2k} - s_{2k-1}} |u_{s_{2k-1}}| \geq 8^{s_{2k} + o(s_{2k})} 
\; .
\]
Therefore, the inequality $|u_n| \geq 8^{n+o(n)}$ holds on a subsequence, hence the sequence $(u_n)$ is not subexponential and neither is $(h_n^2)$.

\begin{rem}
Since in this example $f_n^1$ is diagonal, it can be seen also as an upper triangular linear automorphisms, and one may hope to find a subexponential conjugacy with a sequence $(g_n)$ such that, denoting by $S$ the involution $(z_1,z_2)\mapsto (z_2,z_1)$, $S\circ g_n \circ S$ is a special triangular automorphism of $\C^2$. However, an argument similar to the one developed above shows that any conjugacy $(h_n)$ between $(f_n)$ and a $(g_n)$ of this form has a 2-homogeneous part which diverges exponentially.
\end{rem} 

\begin{rem}
One could show that the basin of attraction of 0 with respect to the above sequence $(f_n)$, that is the set
\[
\set{z\in \C^2}{ f_n  \circ \cdots \circ f_0 (z) \rightarrow 0 \mbox{ for } n\rightarrow \infty}\; ,
\]
is the whole $\C^2$. 
\end{rem}

\section{The abstract basin of attraction}
\label{saba}

Let $\mathscr{G}$ be the category whose objects are the sequences
\[
f = (f_n : U_n \rightarrow U_{n+1})_{n\in \N}
\]
of injective holomorphic  maps between $d$-dimensional complex manifolds and whose morphisms $h\colon f\rightarrow g$ are sequences of injective holomorphic maps 
\[
h = (h_n \colon   U_n \rightarrow V_n)_{n\in \N}\; , \quad \mbox{with } U_n = \dom f_n, \; V_n = \dom g_n\; ,
\]
such that for every $n\in \N$ the diagram
\begin{equation*}
\begin{CD}
U_n @>{f_n}>> U_{n+1} \\ @V{h_n}VV @VV{h_{n+1}}V \\ V_n @>{g_n}>> V_{n+1}
\end{CD}
\end{equation*}
commutes. In other words, $\mathscr{G}$ is the category of functors $\mathrm{Fun} (\N,\mathscr{M})$, where $\mathscr{M}$ is the category of $d$-dimensional complex manifolds and injective holomorphic maps. 

If $f$ is an object of $\mathscr{G}$ and $n\geq m\geq 0 $, we denote by $f_{n,m}$ the composition
\[
f_{n,m} = f_{n-1} \circ \dots \circ f_m \colon  U_m \rightarrow U_n\;, \quad f_{n,n} = \mathrm{id}_{U_n}\;,
\]
which satisfies, for any $n\geq m \geq \ell \geq 0$, 
\[
f_{n,m} \circ f_{m,\ell} = f_{n,\ell}\;.
\]

We denote by $W$ the inductive limit functor 
\[
\mathop{\mathrm{Lim}}_{\longrightarrow} \colon  \mathscr{G} \rightarrow \mathscr{M}\;.
\]
That is, $Wf$ is the topological inductive limit of the sequence of maps $(f_n)$ with the induced holomorphic structure: Constructively, $Wf$ is the quotient of the set
\[
\left\{ z\in \prod_{n\geq m} U_n \; \Big| \; m\in \N\;, \; z_{n+1} = f_n(z_n) \; \forall n\geq m\right\} 
\]  
by identifying $z$ and $z'$ if $z_n=z'_n$ for $n$ large enough. The holomorphic structure is induced by the open inclusions
\[
f_{\infty,m} \colon  U_m \hookrightarrow Wf\;, \quad z \mapsto [(f_{n,m}(z))_{n\geq m}]\;.
\] 
With the above representation, if $h\colon f\rightarrow g$ is a morphism in $\mathscr{G}$, $Wh$ is the injective holomorphic map 
\[
Wh ([(z_n)_{n\geq m}]) = [(h_n(z_n))_{n\geq m}]\;.
\]
The following result, whose proof is immediate, turns our to be useful in order to identify $Wf$:

\begin{lem}
\label{biholo}
Let $h\colon f\rightarrow g$ be a morphism in $\mathscr{G}$. Then $Wh\colon  Wf \rightarrow Wg$ is surjective (hence a biholomorphism) if and only if for every $m\in \N$ and every $z\in V_m$ there exists $n\geq m$ such that $g_{n,m}(z)\in h_n(U_n)$.
\end{lem} 

Let $B=B_1$ be the open unit ball about 0 in $\C^d$ and 
consider a sequence of injective holomorphic maps $f_n \colon  B \rightarrow B$ such that
\begin{equation}
\label{contra}
|f_n(z)| \leq \lambda |z| \;, \qquad \forall z\in B\;, \; \forall n\in \N\;, 
\end{equation}
for some $\lambda<1$. In this case, the manifold $Wf$ may be considered as the {\em abstract basin of attraction of $0$} with respect to the sequence $f=(f_n)$. In fact, if in addiction the maps $f_n$ are restrictions of global automorphisms $g_n$ of $\C^d$, then $g_{\infty,n}\colon \C^d\to Wg$ are biholomorphisms and, in particular, $Wg$ can be identified with $\C^d$; through this identification, the induced holomorphic inclusion $Wf \hookrightarrow Wg\cong \C^d$ is the inclusion in $\C^d$ of the basin of attraction of 0 with respect to $g$, which is the open set
\[
 \set{z\in \C^d}{g_{n,0}(z)\rightarrow 0 \mbox{ for } n \rightarrow \infty}\; ,
\]
and the maps $f_{\infty,n}\colon B \rightarrow Wf$ coincide with $g_{\infty,n}|_B$. Notice also that an immediate application of Lemma \ref{biholo} implies that if $f$ satisfies (\ref{contra}), then for every $r<1$ the manifold $Wf$ is biholomorphic to the abstract basin of attraction of the restriction
\[
f_n|_{B_r} \colon  B_r \rightarrow B_r\;, \qquad n\in \N\; .
\] 

By a {\em bounded sequence of holomorphic germs} we mean a sequence of holomorphic maps
\[
h_n \colon  B_r \rightarrow \C^d\; , \qquad n\in \N\; ,
\]
defined on the same ball of radius $r$ about $0$ and such that $h_n(B_r)$ is uniformly bounded. Under boundedness assumptions, the abstract basin of attraction is invariant with respect to non-autonomous conjugacies, as shown by the following:

\begin{lem}
\label{conimpW}
Let $f=(f_n\colon B \rightarrow B)_{n\in \N}$ and  $g=(g_n\colon B \rightarrow B)_{n\in \N}$ be objects in $\mathscr{G}$ such that
\[
|f_n(z)|\leq \lambda |z|\; , \quad |g_n(z)|\leq \lambda |z| \;, \qquad \forall z\in B\;, \; \forall n\in \N\; ,
\]
for some $\lambda<1$. Assume that there exist $r>0$ and a bounded sequence of holomorphic germs
\[
h_n \colon  B_r \rightarrow \C^d\; , \qquad n\in \N\; ,
\]
such that $h_n(0)=0$, $Dh_n(0)=I$, and
\begin{equation}
\label{conjuqua}
h_{n+1} \circ f_n = g_n \circ h_n  \; , \qquad \forall n\in \N\; ,
\end{equation}
as germs at $0$. Then $Wf$ is biholomorphic to $Wg$.
\end{lem}

\begin{proof}
Fix a positive number $r'<r$.
Since the maps $h_n$ are uniformly bounded on $B_r$, the Cauchy formula implies that the second differential of $h_n$ is uniformly bounded on $B_{r'}$. Since $Dh_n(0)=I$, we deduce that there is a positive number $s\leq r'$ such that $\|Dh_n(z)-I\| <1/2$ for every $z\in B_s$ and every $n\in \N$. In particular, $Dh_n(z)$ is invertible with uniformly bounded inverse for every $z\in B_s$ and every $n\in \N$. Up to the choice of a smaller $s$, we deduce that $h_n$ is a biholomorphism from $B_s$ onto an open subset of $B$ which contains the ball $B_{s'}$, for some $s'>0$ independent on $n$. By (\ref{conjuqua}), the restrictions
\[
h_n|_{B_s} \colon  B_s \rightarrow B\; , \qquad n\in \N\; ,
\]
define a morphism $h$ from the restriction 
\[
(f_n|_{B_s} \colon  B_s \rightarrow B_s)_{n\in \N}
\]
to $g$. Since for every $n,m\in \N$ with $n-m$ large enough 
\[
g_{n,m}(B)\subset B_{\lambda^{n-m}} \subset B_{s'} \subset h_n(B_s)\; ,
\]
Lemma \ref{biholo} implies that $Wh$ is a biholomorphism from the abstract basin of attraction of $(f_n|_{B_s})$ to that of $g$. Since the former manifold is biholomorphic to $Wf$, the conclusion follows.
\end{proof} 

\section{Iteration of special triangular automorphism}
\label{triautosec}

In this section we establish some facts about the composition of special triangular automorphisms. We start by recalling that a family of polynomial endomorphisms of $\C^d$ is said to be {\em bounded} if their degrees and their coefficients are uniformly bounded. The proof of the following lemma is straightforward: 

\begin{lem}
\label{generale}
\begin{enumerate}
\item If $\mathscr{P}$ is a bounded family of polynomial endomorphisms of $\C^d$ such that $p(0)=0$ and $Dp(0)=0$ for every $p$ in $\mathscr{P}$, then the family 
\[
\set{t^{-2} p (t z) }{p\in \mathscr{P}, \; |t|\leq 1} 
\]
is also bounded.
\item If $\mathscr{P}$ and $\mathscr{Q}$ are bounded families of polynomial endomorphisms of $\C^d$, then the family 
\[
\set{p\circ q}{p\in \mathscr{P}, \; q\in \mathscr{Q}} 
\]
is also bounded.
\end{enumerate}
\end{lem}

Let $g_n \colon  \C^d \rightarrow \C^d$ be defined as
\begin{equation}
\label{laforma}
g_n (z) := L_n z + p_n(z)\;,
\end{equation}
where:
\renewcommand{\theenumi}{\alph{enumi}}
\renewcommand{\labelenumi}{(\theenumi)}
\begin{enumerate}
\item $(L_n)$ is a sequence of lower triangular linear automorphisms of $\C^d$
 such that $\|L_{n,m}\|\leq c \lambda^{n-m}$ for every $n\geq m \geq 0$, where $c>0$ and $0<\lambda<1$;
\item $(p_n)$ is bounded as a sequence of polynomial maps of the form
\[
p_n (z_1,\dots, z_d ) = \bigl(0, p_n^2(z_1), \dots, p_n^d (z_1,\dots, z_{d-1} )\bigr)\;,
\]
such that  $p_n(0)=0$, $Dp_n(0)=0$.
\end{enumerate}
In particular, the degree of $p_n$ is bounded, hence there exist positive integers $k_1=1,k_2, \dots,k_d$ such that
\begin{equation}
\label{gradi}
\deg p_n^j (z_1^{k_1},z_2^{k_2}, \dots, z_{j-1}^{k_{j-1}} ) \leq k_j\;, \quad \forall j=2,\dots,d\;.
\end{equation}
The next lemma implies in particular that the composition $g_{n,0}$ has uniformly bounded degree, for every $n\in \N$.

\begin{lem}
\label{grado}
Let $k_1,\dots,k_d$ be positive integers. 
The family endomorphisms of $\C^d$ of the form
\[
f(z) = L z + p(z)
\]
with $L$ varying in the set of lower triangular matrices and $p$ varying among polynomial maps of the form
\[
p(z)=\bigl(0,p_2(z_1),\dots,p_d(z_1,\dots,z_{d-1})\bigr), \quad p(0)=0, \quad Dp(0)=0.
\]
which satisfy
\[
\deg p_j (z_1^{k_1},z_2^{k_2}, \dots, z_{j-1}^{k_{j-1}}) \leq k_j, \quad \forall j=2,\dots,d,
\]
is closed under composition.
\end{lem}

\begin{proof}
If $f(z)=Lz + p(z)$ and $g(z)=M z + q(z)$ are two such maps, then the composition
\[
f\circ g (z) = LM z + L q(z) + p(Mz + q(z))
\]
is readily seen to be of the same form.
\end{proof} 

We can use the above two lemmas to prove the following:

\begin{lem}
\label{limitato}
Let $g_n$ be a sequence of special triangular automorphisms of $\C^d$ of the form (\ref{laforma}) which satisfies (a) and (b).
Then the family of polynomial maps $\lambda^{-n} g_{n,0}$, $n\in \N$, is bounded.
\end{lem}

\begin{proof}
We argue by induction on the dimension $d$. If $d=1$, the composition $g_{n,0}=L_{n,0}$ is linear and the conclusion follows from the estimate in the assumption (a). We assume that the conclusion holds when the dimension is less than $d$.
By Lemma \ref{generale} (i), the sequence of polynomial maps
\[
\left( \lambda^{-2n} p_n (\lambda^n z) \right)
\] 
is bounded. Denote by $g_{n,0}^i$ the $i$-th component of the map $g_{n,0}$.
By the inductive hypothesis, the sequence $(\lambda^{-n} g_{n,0}^i)$, for $1\leq i\leq d-1$, is bounded. Therefore, by Lemma \ref{generale} (ii), the sequence
\begin{equation}
\label{num1}
\lambda^{-2n} p_n \circ g_{n,0} = \lambda^{-2n} p_n \bigl( \lambda^{n} (\lambda^{-n} g_{n,0}^1 ), \dots, \lambda^{n} (\lambda^{-n} g_{n,0}^{d-1}) \bigr)
\end{equation}
is also bounded. Since
\[
g_{n+1,0} = L_n g_{n,0} + p_n\circ g_{n,0} \quad \mbox{and} \quad g_{0,0} = \mathrm{id}\;,
\]
by the general formula (\ref{solution}), we have
\[
g_{n,0}  = L_{n,0} + \sum_{j=0}^{n-1} L_{n,j+1} p_j \circ g_{j,0}\;,
\]
which can also be written as
\[
\lambda^{-n} g_{n,0} = \lambda^{-n} L_{n,0} + \sum_{j=0}^{n-1} \lambda^j \bigl( \lambda^{-n} L_{n,j+1} \bigr) \bigl( \lambda^{-2j} p_j \circ g_{j,0} \bigr)\;.
\]
The estimate of the assumption (a) and the fact that (\ref{num1}) is bounded imply that $\lambda^{-n} g_{n,0}$ is bounded.
\end{proof}

We conclude this section by deducing the following:

\begin{lem}
\label{triang}
Let $g_n$ be a sequence of special triangular automorphisms of $\C^d$ of the form (\ref{laforma}) which satisfies (a) and (b) and let $k:= \max\{k_1,\dots,k_d\}$, where the numbers $k_1=1,k_2,\dots,k_d$ satisfy (\ref{gradi}). Then there exists a number $C$ such that
\[
|g_{n,0} (z)| \leq C \lambda^n ( |z| + |z|^k)\;, \qquad \forall z\in \C^d\;,
\]
for every $n\in \N$. In particular, the basin of attraction of $0$ of the sequence $(g_n)$ is the whole $\C^d$:
\[
 \set{z\in \C^d}{g_{n,0}(z)\rightarrow 0 \mbox{ for } n \rightarrow \infty}= \C^d\;.
\]
\end{lem}

\begin{proof}
By Lemmas \ref{grado} and \ref{limitato}, $(\lambda^{-n} g_{n,0})$ is a bounded sequence of polynomial maps of degree at most $k$, mapping $0$ into $0$. Since $\lambda<1$, the claim immediately follows.
\end{proof}

\section{Proofs of the main results}

We are now ready to prove the first main theorem stated in the Introduction:

\begin{thm}
\label{basin}
Let $B$ be the open unit ball of $\C^d$ and let $f=(f_n\colon B\rightarrow B)_{n\in \N}$ be an element of $\mathscr{G}$ such that  for every $n\in \N$
\begin{equation}
\label{lip}
\nu |z| \leq |f_n(z)|\leq \lambda |z| \;, \qquad \forall z\in B\;,
\end{equation}
with $0< \nu \leq \lambda<1$. Assume that the linear automorphism $L_n := Df_n(0)$ is lower triangular and satisfies the condition (\ref{ord}) of Theorem \ref{formal}.
Then the abstract basin of attraction $Wf$ is biholomorphic to $\C^d$.
\end{thm}

\begin{proof}
Notice that by (\ref{lip}), $\|L_n\| \leq \lambda$ and $\|L_n^{-1}\|  \leq \mu := 1/\nu$, so the assumption (\ref{stime}) of Theorem \ref{formal} holds, hence also (\ref{lambdamu}) does (with a suitable $m_0$). 
Since the holomorphic maps $f_n$ are all defined on the unit ball $B$, map 0 into 0, and have uniformly bounded images, by considering their Taylor expansion at 0 we may see $(f_n)$ as a bounded sequence in the space of formal series $\mathscr{F}$. In particular, $(f_n)$ is a subexponential sequence such that $f_n^1=L_n$ satisfies the assumptions of Theorem \ref{formal}, which implies the existence of a subexponential sequence $(g_n)$ of  special triangular automorphisms of $\C^d$ of degree at most $m_0$ and a subexponential sequence $(h_n)$ in $\mathscr{F}$ such that  $g_n^1=L_n$, $h_n^1=I$, and
\begin{equation}
\label{conjnorm}
h_{n+1} \circ f_n = g_n \circ h_n \; , \qquad \forall n\in \N\; .
\end{equation}
By (\ref{lambdamu}), we can find $1<\theta<1/\lambda$ such that
\begin{equation}
\label{tpic}
\theta^{m_0} \lambda^{m_0+1} \mu<1\; .
\end{equation}
For every $n\in \N$, consider the map
\[
\tilde{f}_n(z) := \theta^{n+1} f_n\bigl(\theta^{-n} z\bigr)\;, \qquad \forall z\in B\; .
\]
The linear automorphisms
\[
\tilde{L}_n := D\tilde{f}_n(0)  = \theta L_n
\]
satisfy the bounds
\begin{equation}
\label{linbds}
\|\tilde{L}_n\|  \leq \tilde{\lambda}\; , \quad \|\tilde{L}_n^{-1}\|  \leq \tilde{\mu}\; ,
\end{equation}
where
\[
\tilde{\lambda} := \theta \lambda< 1\;, \quad \tilde{\mu} := \theta^{-1} \mu\;.
\]
By (\ref{lip}), the Cauchy formula implies that $Df_n$ is uniformly bounded on $B_{1/2}$. Therefore
\[
D\tilde{f}_n(z) = \theta Df_n(\theta^{-n} z)
\]
is uniformly bounded on $B_{1/2}$, hence $(\tilde{f}_n)$ is a bounded sequence of germs. Fix a number $\hat{\lambda}$ such that $\tilde{\lambda}< \hat{\lambda} < 1$.
Together with the first of the bounds in (\ref{linbds}), a further use of the Cauchy formula implies that there exists $0<r\leq 1$ such that for every $n\in \N$ 
\begin{equation}
\label{qw1}
|\tilde{f}_n(z)|\leq \hat{\lambda} |z|\; ,\quad \forall z\in B_r\;.
\end{equation}
In particular, 
\[
\tilde{f} := (\tilde{f}_n|_{B_r} : B_r \rightarrow B_r)_{n\in \N}
\]
can be seen as an object of $\mathscr{G}$. The maps 
\[
\varphi_n\colon B_r \rightarrow B\;, \quad \varphi_n(z):= \theta^{-n} z\; , 
\]
define a morphism $\varphi\colon  \tilde{f} \rightarrow f$, which induces a holomorphic injection $W\varphi\colon  W\tilde{f} \rightarrow Wf$. By (\ref{lip}) and since $\lambda<\theta^{-1}$, for every $m\in \N$ there is a natural number $n\geq m$ so large that
\[ 
f_{n,m}(B) \subset B_{\lambda^{n-m}} \subset B_{r\theta^{-n}} = h_n(B_r)\;.
\]
Hence, Lemma \ref{biholo} implies that $W\varphi$ is a biholomorphism. Therefore, it is enough to show that $W\tilde{f}$ is biholomorphic to $\C^d$. 

The same rescaling used above defines the polynomial maps of degree at most $m_0$
\[
\tilde{g}_n(z) := \theta^{n+1} g_n\bigl(\theta^{-n} z\bigr)\; ,
\]
which satisfy
\[
\tilde{g}_n(0)=0\;, \quad D\tilde g_n(0) = \tilde{L}_n\; .
\] 
If $2\leq k \leq m_0$, the homogeneous part of degree $k$ of $\tilde{g}_n$ is
\[
\tilde{g}_n^k =  \theta^{(1-k)n+1} g_n^k\; ,
\]
so the fact that $(g_n^k)$ is subexponential and $\theta>1$ imply that the sequence of polynomials $(\tilde{g}_n)$ is bounded. Up to the choice of a smaller $r>0$, by the first bound in (\ref{linbds}) and the Cauchy formula, we may assume that
\begin{equation}
\label{qw2}
|\tilde{g}_n(z)| \leq \hat{\lambda} |z|\; , \quad \forall z\in B_r\; .
\end{equation}
Similarly, the sequence $(\tilde h_n)\subset \mathscr{F}$ of formal series defined by
\[
\tilde{h}_n (z) := \theta^n h_n\bigl( \theta^{-n} z \bigr)\; ,
\]
is bounded in $\mathscr{F}$. By (\ref{conjnorm}) and by the definition of $\tilde{f}_n$, $\tilde{g}_n$ and $\tilde{h}_n$,
\begin{equation}
\label{conjutilde}
\tilde{h}_{n+1} \circ \tilde{f}_n = \tilde{g}_n \circ \tilde{h}_n\;, \qquad \forall n\in \N\; ,
\end{equation}
as formal series, so in particular as $m_0$-jets.
Since by (\ref{tpic})
\[
\tilde{\lambda}^{m_0+1} \tilde{\mu} = \theta^{m_0} \lambda^{m_0+1} \mu<1\; ,
\]
by the bounds (\ref{linbds}) Theorem \ref{volata} implies that (\ref{conjutilde}) is satisfied as an identity of germs by a bounded sequence of germs $(\tilde{h}_n)$. By (\ref{qw1}), (\ref{qw2}) and (\ref{conjutilde}), Lemma \ref{conimpW} implies that $W\tilde{f}$ is biholomorphic to $W\tilde{g}$.
Since
\[
W\tilde{g} \cong \set{z\in \C^d}{\tilde{g}_{n,0}(z)\rightarrow 0\mbox{ for } n\rightarrow \infty} = \C^d\; ,
\]
by Lemma \ref{triang} we conclude that $W\tilde{f}$ is biholomorphic to $\C^d$ and hence so is $Wf$.
\end{proof} 

\begin{rem}
Instead than applying Theorem \ref{volata}, which establishes the existence of a bounded non-autonomous conjugacy of arbitrary sequences of germs starting from a conjugacy between their jets, one could use the fact that in the above case one of the two germs consists of special triangular automorphisms. In such a case, the convergent conjugacy is easier to obtain, by an argument due to J.-P. Rosay and W. Rudin \cite{rr88}, and used also in the already mentioned \cite{jv02} and \cite{pet07}.
\end{rem}

Finally, let us use the above theorem to deduce the second main theorem stated in the Introduction: 

\begin{cor}
\label{ucurullariu}
Let $f\colon M\rightarrow M$ be a holomorphic automorphism of a complex manifold and let $\Lambda$ be a compact hyperbolic invariant set with $d$-dimensional stable distribution $E^s$. Let $x\in \Lambda$ and assume that the stable space at $x$ has a splitting
\[
E^s(x) = \bigoplus_{i=1}^r E_i 
\]
such that
\[
\lim_{n\rightarrow \infty} |Df^n(x)\, u|^{1/n} = \bar{\lambda}_i \quad \mbox{uniformly for $u$ in the unit sphere of } E_i \; ,
\]
where the numbers $0<\bar\lambda_i<1$ are pairwise distinct. Then the stable manifold of $x$ is biholomorphic to $\C^d$.
\end{cor} 

\begin{proof}
Set $x_n=f^n(x)$. By the local stable manifold theorem and up to the replacement of $f$ with a sufficiently high iterate (an operation which does not change the stable manifold), we can find holomorphic embeddings
\[
\varphi_n \colon  B \hookrightarrow W^s(x_n)
\]
with domain the unit ball about $0$ in $\C^d$, mapping $0$ into $x_n$, and such that the identities
\[
f \circ \varphi_n = \varphi_{n+1} \circ f_n
\]
define holomorphic maps $f_n\colon B\rightarrow B$ such that
\begin{equation}
\label{lipp}
\nu |z| \leq |f_n(z)| \leq \lambda |z| \; , \qquad \forall z\in B, \; \forall n\in \N\; ,
\end{equation}
for some $0 < \nu \leq \lambda<1$. A biholomorphism from the stable manifold $W^s(x)$ onto 
the abstract basin of attraction of $0$ with respect to the sequence $(f_n)$ is given by mapping each $z\in W^s(x)$ into the equivalence class of the sequence $(f^n(z))_{n\geq m}$, where $m$ is so large that $f^m(z)$ belongs to the image of $\varphi_m$. Therefore, it is enough to show that this abstract basin of attraction is biholomorphic to $\C^d$.

Since the angles between the images of the subspaces $E_i$ by the isomorphism $Df^n(x_0)$ remain bounded away from zero (see \cite[Section IV.11]{man83}), by using a suitable linear non-autonomous conjugacy we may also assume that the automorphisms
\[
L_n = Df_n(0) 
\]
are lower triangular and preserve the orthogonal splitting of $\C^d$
\[
\C^d = \bigoplus_{i=1}^r X_i\; , \qquad X_i = \mathrm{Span}\,  \set{e_j}{k_{i-1} < j \leq k_i} \; ,
\]
for some $0=k_0 < k_1 < \dots < k_r = d$, and
\begin{equation}
\label{limi}
\lim_{n\rightarrow \infty} | L_{n,0}\, u|^{1/n} = \bar\lambda_i \quad \mbox{uniformly for $u$ in the unit sphere of }X_i\; ,
\end{equation}
where 
\[
\bar\lambda_1 >  \bar\lambda_2 > \dots > \bar\lambda_r \; .
\]
Denote by $\lambda_n(j)$, $1\leq j \leq d$, the diagonal entries of $L_n$. Fix an index $j$ such that $k_{i-1}<  j \leq k_i$.  Since $\lambda_{n,0}(j)$ is an eigenvalue of $L_{n,0}$ with eigenvector $u_n$ in the unit sphere of $X_i$, by (\ref{limi}) we have
\[
\lim_{n\rightarrow \infty} |\lambda_{n,0}(j)|^{1/n} = \lim_{n\rightarrow \infty} |L_{n,0} u_n|^{1/n} =\bar\lambda_i\;.
\]
If we set $\hat\lambda_j = \bar\lambda_i$ for $k_{i-1}<  j \leq k_i$, we have that
\[
\hat\lambda_1 \geq \hat\lambda_2 \geq \dots \geq \hat\lambda_d\; ,
\]
and for every $\omega>1$ there exists $c(\omega)>0$ such that
\[
c(\omega)^{-1} \omega^{-n}  \hat{\lambda}_j^n \leq  |\lambda_{n,0}(j)| \leq c(\omega) \omega^n \hat{\lambda}_j^n \; .
\]
The above two inequalities imply that, if $1\leq i \leq j \leq d$ and $n,\ell\in \N$, then
\[
\frac{|\lambda_{n+\ell,n} (j)|}{|\lambda_{n+\ell,n} (i)|} = \frac{  |\lambda_{n+\ell,0} (j)| \, |\lambda_{n,0} (i)|}{ |\lambda_{n,0} (j)| \, |\lambda_{n+\ell,0} (i)|} \leq \frac{ c(\omega)^2 \omega^{2n+\ell} \hat{\lambda}_j^{n+\ell} \hat{\lambda}_i^n}{ c(\omega)^{-2} \omega^{-2n-\ell}  \hat{\lambda}_j^n \hat{\lambda}_i^{n+\ell} } = c(\omega)^4 \omega^{4n+2\ell} \left( \frac{\hat\lambda_j}{\hat\lambda_i} \right)^{\ell} \leq c(\omega)^4 \omega^{4n+2\ell}
\]
Moreover, by (\ref{lipp}), $|\lambda_n(h)|\leq \lambda$ for every $n\in \N$, thus
\[
\max_{1\leq h \leq d} |\lambda_{n+\ell,n}(h)| \max_{1\leq i \leq j \leq d} \frac{|\lambda_{n+\ell,n}(j)|}{|\lambda_{n+\ell,n}(i)|} \leq c(\omega)^4 \omega^{4n} (\omega^2 \lambda)^{\ell}.
\]
This inequality shows that the assumption (\ref{ord}) of Theorem \ref{formal} holds (with a slightly larger $\lambda$), hence Theorem \ref{basin} implies that the abstract basin of attraction of $0$ with respect to $(f_n)$ is biholomorphic to $\C^d$.
\end{proof}

\renewcommand{\thesection}{\Alph{section}}
\setcounter{section}{0}
\setcounter{equation}{0}
\numberwithin{equation}{section}

\section{A non-autonomous conjugacy theorem}

The aim of this appendix is to prove the non-autonomous version of the well-known fact that two contracting germs of holomorphic maps which are conjugated as jets of a sufficiently high degree are actually conjugated as germs. Our proof follows the approach of Sternberg \cite{ste57}, but we replace the delicate estimates which are necessary to apply the Banach contraction principle by an easier computation of the spectral radius of the linearized operator, which allows us to apply the implicit function theorem. See \cite{bdm08} for a different approach.

The space $\C^d$ is endowed with the norm
\[
|z| := \max_{1\leq j \leq d} |z_j|\; , \qquad \forall z=(z_1,\dots,z_d) \in \C^d\; ,
\]
whose ball of radius $r$ about $0$, that we denote by $B_r$, is an open  polydisk. If $L$ is a linear endomorphism of $\C^d$, $\|L\|$ indicates its operator norm induced by $|\cdot|$. If $(f_n)$ is a sequence of composable maps we set, for $n\geq m \geq 0$,
\[
f_{n,m} := f_{n-1} \circ f_{n-2} \circ \cdots \circ f_m\;, \quad f_{n,n}=I\; ,
\]
so that
\[
f_{n,m} \circ f_{m,\ell} = f_{n,\ell}\;, \qquad \forall n\geq m \geq \ell \geq 0\; .
\]
Denote by $\mathscr{G}$ the space of germs at $0\in \C^d$ of holomorphic $\C^d$-valued maps which fix $0$. A sequence $(f_n)\subset \mathscr{G}$ is said to be \emph{bounded} if there exists $r>0$ such that $B_r$ is contained in the domain of each $f_n$ and $f_n(B_r)$ is uniformly bounded.   

\begin{thm}
\label{volata}
Let $(f_n)$ and $(g_n)$ be two bounded sequences in $\mathscr{G}$, whose linear parts coincide,
\[
L_n := Df_n(0) = Dg_n(0) \; ,
\]
and satisfy
\begin{equation}
\label{contrae}
\|L_{n,m}\|\leq c \lambda^{n-m}\;, \quad \|L_{n,m}^{-1}\| \leq c \mu^{n-m}\;, \qquad \forall n\geq m\in \N\;,
\end{equation}
for some $c>0$ and $0<\lambda<1<\mu$. Let $k$ be a positive integer such that
\[
\lambda^{k+1} \mu < 1\;.
\]
Assume that the $k$-jets of $(f_n)$ and $(g_n)$ are boundedly conjugated, meaning that there exists a bounded sequence of polynomial maps $H_n \colon  \C^d \rightarrow \C^d$, $n\in \N$, of degree at most $k$, such that
\begin{equation}
\label{kconju}
H_{n+1} \circ f_n  =  g_n \circ H_n \quad \mbox{as }k \mbox{-jets, } \qquad \forall n\in \N\;.
\end{equation}
Then $(f_n)$ and $(g_n)$ are boundedly conjugated as germs: There exists a bounded sequence $(h_n)\subset \mathscr{G}$ such that the $k$-jet of $h_n$ is $H_n$ and 
\begin{equation}
\label{gconju}
h_{n+1} \circ f_n = g_n \circ h_n\; , \qquad \forall n\in \N\; ,
\end{equation}
as germs. 
\end{thm}

As already noticed (see Remark \ref{indep}), the assumption (\ref{contrae}) does not depend on the choice of the norm of $\C^d$ inducing the operator norm, up to the choice of a different constant $c$. Moreover, we may replace this assumption by the stronger requirement
\begin{equation}
\label{contrae+}
\|L_n \| \leq \lambda<1 \; , \quad \|L_n^{-1}\|\leq \mu\;, \qquad \forall n\in \N\; .
\end{equation}
Indeed, let $\hat{\lambda}$ and $\hat{\mu}$ be such that $\lambda<\hat{\lambda}<1$, $\hat{\mu}>\mu$ and $\hat{\lambda}^{k+1}\hat{\mu}<1$. Then (\ref{contrae}) implies that we can find a natural number $N$ such that
\[
\|L_{n+N,n}\| \leq  c\lambda^N \leq \hat{\lambda}^N\; , \quad \|L_{n+N,n}^{-1} \| \leq c\mu^N \leq \hat{\mu}^N \; , \qquad \forall n\in \N\; .
\]
If we set $\tilde{L}_n:=L_{(n+1)N,nN}$,
$\tilde{\lambda}:= \hat\lambda^N<1$, $\tilde{\mu}:= \hat{\mu}^N$, we see that the sequence $(\tilde{L}_n)$ fulfills the condition (\ref{contrae+}) with constants $\tilde{\lambda}$ and $\tilde{\mu}$ which satisfy 
\[
\tilde{\lambda}^{k+1} \tilde \mu = (\hat\lambda^{k+1} \hat{\mu})^N < 1\; .
\]
If we set
\[
\tilde{f}_n := f_{(n+1)N,nN}\;, \quad \tilde{g}_n := g_{(n+1)N,nN}\;,
\]
and we define $\tilde{H}_n$ to be the $k$-jet of the composition $H_{(n+1)N,nN}$, we get that $(\tilde{f}_n)$, $(\tilde{g}_n)$ and $(\tilde{H}_n)$ 
satisfy the assumptions of the theorem, with (\ref{contrae}) replaced by the stronger (\ref{contrae+}). Finally, if $(\tilde{h}_n)$ is a bounded conjugacy between $(\tilde{f}_n)$ and $(\tilde{g}_n)$ with sequence of $k$-jets $(\tilde{H}_n)$, then the unique sequence $(h_n)\subset \mathscr{G}$ such that $h_{nN} = \tilde{h}_n$ 
for every $n\in \N$ and such that (\ref{gconju}) holds, is bounded and has $(H_n)$ as its sequence of $k$-jets. This shows that we may assume (\ref{contrae+}) instead of (\ref{contrae}).

\medskip

Denote by $\|\cdot\|_{\infty,r}$ the supremum norm on the polydisk $B_r$.  
Consider the space 
\[
X_k := \set{u: B_1 \rightarrow \C^d}{u \mbox{ holomorphic and bounded with vanishing }k\mbox{-jet at }0}\; .
\]
By the Cauchy formula, $\|\cdot\|_{\infty,1}$ is a Banach norm on $X_k$. Moreover, using the Taylor formula with integral remainder and again the Cauchy formula, we find that for every $u\in X_k$ and every $r<1$ there holds
\begin{equation}
\label{kpiatta}
\|u\|_{\infty,r} \leq \frac{r^{k+1}}{(k+1)!} \|D^{k+1} u\|_{\infty,r} \leq C \left( \frac{r}{1-r} \right)^{k+1} \|u\|_{\infty,1}\; ,
\end{equation}
where $C=C(d,k)$ depends only on the dimension $d$ and on the degree $k$. 

Let $\ell^{\infty}(X_k)$ be the Banach space of sequences $u=(u_n)\subset X_k$ with finite supremum norm
\[
\|u\|_{\ell^{\infty}(X_k)} := \sup_{n\in \N} \|u_n\|_{\infty,1} < +\infty,
\]  
and denote by $U$ its unit open ball.
If $f\colon B_s \rightarrow \C^d$ is a holomorphic map such that $f(0)=0$, we consider the rescaled functions
\[
f^r \colon  B_1 \rightarrow \C^d\;, \quad f^r (z) := \left\{ \begin{array}{ll} \displaystyle{ \frac{1}{r} f(rz)} & \mbox{if } 0<r\leq s\;, \\ \displaystyle{ Df(0) z} & \mbox{if } r=0\;. \end{array} \right.
\]
In order to prove Theorem \ref{volata}, it is enough to find a sequence $u\in U$ such that for $r>0$ small enough there holds
\begin{equation}
\label{tesi}
(H_{n+1}^r + u_{n+1} \bigr) \circ f_n^r - g_n^r \circ ( H_n^r + u_n ) = 0\; , \qquad \forall n\in \N\; .
\end{equation}
In fact, in such a case the sequence of holomorphic functions
\[
h_n\colon  B_r \rightarrow \C^d, \quad h_n(z):=  H_n(z) + r\, u_n \bigl( z/r \bigr)\;,
\]
is bounded, has $H_n$ as its $k$-jet  and satisfies (\ref{gconju}).

By (\ref{kconju}), the $k$-th jet of the left-hand side of (\ref{tesi}) vanishes. Moreover, the fact that the sequences $(f_n)$ and $(H_n)$ are bounded and $\|Df_n(0)\| \leq \lambda <1$ impy that there exists $r_0>0$ such that for every $r\in ]0,r_0[$ and every $u\in U$, there holds
\begin{eqnarray*}
& f_n^r(B_1) = \frac{1}{r} f_n(B_r) \subset B_{\|Df_n\|_{\infty,r}} \subset B_1\; ,& \\ &
\bigl( H_n^r + u_n\bigr) (B_1) \subset \frac{1}{r} H_n(B_r) + u_n(B_1) \subset B_{\|DH_n\|_{\infty,r} + 1} \subset B_{1/r}\; . &
\end{eqnarray*}
The above considerations imply that the map
\[
\Phi\colon  [0,r_0[ \times U\rightarrow \ell^{\infty}(X_k)\; , \quad \Phi(r,u)_n :=  ( H_{n+1}^r + u_{n+1} ) \circ f_n^r - g_n^r \circ  (H_n^r + u_n )\; ,
\]
is well-defined. 
Notice that for $r=0$ the map $\Phi$ is linear in $u$,
\begin{equation}
\label{zero}
\Phi(0,u) = \bigl(u_{n+1} \circ L_n - L_n \circ u_n\bigr)_{n\in \N}\; .
\end{equation}

Our aim is to show that if $r>0$ is small enough, then there exists $u\in U$ such that $\Phi(r,u)=0$. Since $\Phi(0,0)=0$, this fact is an immediate consequence of the parametric inverse mapping theorem, because of the following two Lemmas.

\begin{lem}
The map $\Phi$ is continuous, is differentiable with respect to the second variable, and 
\[
D_2 \Phi \colon  [0,r_0[ \times U \rightarrow L\bigl( \ell^{\infty}(X_k) \bigr) 
\]
is continuous.
\end{lem}

\begin{proof}
We shall repeatedly make use of the following consequence of Taylor's and Cauchy's formulas: If $F\subset \mathscr{G}$ is a bounded set of germs whose $h$-jets vanish, then
\begin{equation}
\label{vanh} 
\|D^j f^r\|_{\infty,s} = O(r^{h-j}) \quad \mbox{for } r\rightarrow 0 \mbox{ , uniformly in } f\in F\;,
\end{equation}
for every $j\in \N$ and $s>0$. By (\ref{kconju}),
\[
\ell_n := H_{n+1} \circ f_n - g_n \circ H_n
\]
is a bounded sequence of germs with vanishing $k$-jet. Write
\[
f_n = L_n + \hat{f}_n \; , \quad  g_n = L_n + \hat{g}_n \; , 
\]
where $(\hat{f}_n)$ and $(\hat{g}_n)$ are bounded sequences in $\mathscr{G}$ with vanishing 1-jet. If $u\in U$ and $0<r<r_0$, a simple computation yields
\begin{equation}
\label{effe}
\bigl[ \Phi(r,u) - \Phi(0,u) \bigr]_n = \ell_n^r + \hat{g}_n^r \circ H_n^r - \hat{g}_n^r \circ ( H_n^r + u_n ) + u_{n+1}\circ f_n^r - u_{n+1}\circ L_n.
\end{equation}
Since the $k$-jet of $\ell_n$ is zero and $k\geq 1$, (\ref{vanh}) implies that
\begin{equation}
\label{f0}
\| \ell_n^r \|_{\infty,1} = O(1) .
\end{equation}
Here and in the following lines, limits are for $r\rightarrow 0$ and are uniform in $n\in \N$.
Moreover $H_n(0)=0$, so by (\ref{vanh})
\[
\| H_n^r \|_{\infty,1} = O(1)\; .
\]
Then, since the 1-jet of $\hat{g}_n$ vanishes, a further use of (\ref{vanh}) implies that
\begin{equation}
\label{f1}
\| \hat{g}_n^r \circ H_n^r  \|_{\infty,1} = O(r) \;.
\end{equation}
Similarly,
\begin{equation}
\label{f2}
\bigl\| \hat{g}_n^r \circ (H_n^r + u_n ) \bigl\|_{\infty,1} = O(r)\;.
\end{equation}
Again by (\ref{vanh}),
\[
\| f_n^r - L_n \|_{\infty,1} = O(r)\; ,
\]
so, by the mean value theorem and (\ref{contrae+}),
\[
\| u_{n+1} \circ f_n^r - u_{n+1} \circ L_n \|_{\infty,1} = \|D u_{n+1} \|_{\infty,\lambda+O(r)} O(r) \; .
\]
Since $\lambda<1$, the Cauchy formula implies that $\|D u_{n+1} \|_{\infty,\lambda+O(r)}$ is bounded by $\|u_{n+1}\|_{\infty,1}$ for $r$ small enough, so we have
\begin{equation}
\label{f3}
\| u_{n+1} \circ f_n^r - u_{n+1}\circ L_n \|_{\infty,1} = \|u_{n+1}\|_{\infty,1} O(r) \; .
\end{equation}
Formula (\ref{effe}) and the asymptotics (\ref{f0}), (\ref{f1}), (\ref{f2}), (\ref{f3}) imply that for every $u\in U$ 
\[
\Phi(r,u) - \Phi(0,u) = O(r)\; .
\]
Together with the fact that the map $\Phi$ is clearly continuous on $]0,r_0[ \times U$, this proves that $\Phi$ is continuous on $[0,r_0[ \times U$. 

The map $u\mapsto \Phi(r,u)$ is linear for $r=0$; for $r>0$ its differential is given by
\[
\bigl[ D_2 \Phi(r,u) [v] \bigr]_n = v_{n+1} \circ f_n^r - D g_n^r (H_n^r+u_n) [v_n]\; .
\]
Then if $0<r<r_0$,
\begin{equation}
\label{effe1}
\bigl[ D_2 \Phi(r,u)[v] - D_2 \Phi(0,u) [v] \bigr]_n = v_{n+1} \circ f_n^r - v_{n+1} \circ L_n - D g_n^r (H_n^r + u_n) [v_n] - L_n v_n\; .
\end{equation}
By (\ref{f3}),
\begin{equation}
\label{f4}
\| v_{n+1} \circ f_n^r - v_{n+1} \circ L_n \|_{\infty,1} = \|v_{n+1}\|_{\infty,1} O(r)\; .
\end{equation}
Moreover, the identity
\[
D g_n^r (H_n^r + u_n) = D g_n ( r H_n^r + r u_n)
\]
implies that
\[
\bigl\| D g_n^r (H_n^r+u_n) - L_n \bigr\|_{\infty,1} = o(1)\; ,
\]
hence
\begin{equation}
\label{f5}
\bigl\| D g_n^r ( H_n^r + u_n)[v_n] - L_n v_n \|_{\infty,1} = \|v_n\|_{\infty,1} o(1)
\; .
\end{equation} 
By (\ref{effe1}), (\ref{f4}) and (\ref{f5}), 
\[
\bigl\| D_2 \Phi(r,u) - D_2 \Phi (0,u) \bigr\| = o(1)\; ,
\]
and together with the fact that the map $(r,u) \mapsto D_2 \Phi(r,u)$ is easily seen to be continuous on $]0,r_0[ \times U$, we conclude that this map is continuous on $[0,r_0[ \times U$.
\end{proof}

\begin{lem}
The linear operator $D_2 \Phi(0,0)\colon  \ell^{\infty}(X_k) \rightarrow \ell^{\infty}(X_k)$ is an isomorphism.
\end{lem}

\begin{proof}
By (\ref{zero}), we have
\[
D_2\Phi(0,0)[u] = \bigl( u_{n+1} \circ L_n - L_n \circ u_n \bigr)_{n\in \N} \;.
\] 
The multiplication operator by the sequence $(L_n)$, that is
\[
(u_n) \mapsto (L_n u_n),
\]
is an automorphism of $\ell^{\infty}(X_k)$, so it is enough to show that the operator 
\begin{equation}
\label{ope}
(u_n) \mapsto (L_n^{-1} \circ u_{n+1} \circ L_n - u_n)
\end{equation}
is invertible on $\ell^{\infty}(X_k)$. Consider the bounded linear operator
\[
T\colon  \ell^{\infty}(X_k) \rightarrow \ell^{\infty}(X_k), \quad (u_n) \mapsto (L_n^{-1} \circ u_{n+1} \circ L_n)\;.
\]
If $j$ is a positive integer, the $j$-th power of $T$ is
\[
(T^j u)_n = L_{n,n+j} \circ u_{n+j} \circ L_{n+j,n}\;.
\]
By (\ref{contrae+}) and (\ref{kpiatta}), we have the estimate
\begin{eqnarray*}
\|(T^j u)_n\|_{\infty,1} = \|L_{n+j,n}^{-1} \circ u_{n+j} \circ L_{n+j,n}\|_{\infty,1} \leq \mu^j \|u_{n+j} \circ L_{n+j,n}\|_{\infty,1} \\
\leq \mu^j \|u_{n+j}\|_{\infty,\lambda^j} \leq C \mu^j \left( \frac{\lambda^j}{1-\lambda^j} \right)^{k+1} \|u_{n+j}\|_{\infty,1} \leq C (1-\lambda^j)^{-k-1} (\lambda^{k+1} \mu)^j \|u\|_{\ell^{\infty}(X_k)}\;.
\end{eqnarray*}
By taking the supremum over all $n\in \N$ and all $u\in \ell^{\infty}(X_k)$, we deduce that the operator norm of $T^j$ has the upper bound
\[
\|T^j\|^{1/j} \leq C^{1/j} (1-\lambda^j)^{-(k+1)/j} \lambda^{k+1} \mu\;.
\]
Since $\lambda<1$, the right-hand side of the above inequality converges to $\lambda^{k+1} \mu$ for $j\rightarrow +\infty$, so the spectral radius of $T$ has the upper bound
\[
\rho(T) \leq \lambda^{k+1} \mu\;.
\]
Since $\lambda^{k+1} \mu<1$, we deduce in particular that 1 does not belong to the spectrum of $T$, so the operator (\ref{ope}) is an isomorphism, as we wished to prove.
\end{proof}  

\begin{rem}
\label{domain}
We recall that any polynomial map $p\colon \C^d \rightarrow \C^d$ of degree at most $k$ with $Dp(0)$ invertible  is the $k$-jet at $0$ of a holomorphic automorphism of $\C^d$ (as proven by F. Forstneric \cite{for99} and B. Weickert \cite{wei97}). This fact and  
Theorem \ref{volata} immediately imply that the abstract basin of attraction (see Section \ref{saba}) of $0$ with respect to a sequence of holomorphic maps $f_n \colon  B_1 \rightarrow B_1$  such that 
\[
\nu |z| \leq |f_n(z)| \leq \lambda |z|\;, \qquad \forall z\in B_1\;, \; \forall n\in \N\;,
\]
where $0<\nu\leq\lambda<1$ is always biholomorphic to a domain of $\C^d$, a result of J.E. Fornaess and B. Stens{\o}nes, see \cite{fs04}. Indeed, the above assumption implies that
\[
f_n(0)=0\;, \quad \|Df_n(0)\|\leq \lambda \; , \quad \|Df_n(0)^{-1}\| \leq \nu^{-1}\; ,
\]
and if $g_n$ is an automorphism of $\C^d$ whose $k$-jet at $0$ coincides with that of $f_n$, with $k$ so large that $\lambda^{k+1} \nu^{-1}<1$, Theorem \ref{volata} and Lemma \ref{conimpW} imply that the abstract basins of attraction of $0$ with respect to $(f_n)$ and to $(g_n)$ are biholomorphic. But, since each $g_n$ is an automorphism of $\C^d$, the abstract basin of attraction of $0$ with respect to $(g_n)$ is naturally biholomorphic to the following domain in $\C^d$
\[
\set{z\in \C^d}{g_{n,0}(z) \rightarrow 0\mbox{ for } n\rightarrow \infty}\; ,
\]
as shown in Section \ref{saba}. In particular, the stable manifold of any point in a compact hyperbolic invariant set of a holomorphic automorphism of a complex manifold is always biholomorphic to a domain in $\C^d$.
\end{rem}

%\bibliographystyle{amsalpha}
%\bibliography{../../biblio/nonlinear}

\begin{thebibliography}{PVW08}

\bibitem[BDM08]{bdm08}
F.~Berteloot, C.~Dupont, and L.~Molino, \emph{Normalization of bundle
  holomorphic contractions and applications to dynamics}, Ann. Inst. Fourier
  \textbf{58} (2008), 2137--2168.

\bibitem[Bed00]{bed00}
E.~Bedford, \emph{Open problem session of the biholomorphic mappings meeting at
  the american institute of mathematics}, Palo Alto, CA, July 2000.

\bibitem[For99]{for99}
F.~Forstneric, \emph{Interpolation by holomorphic automorphisms and embeddings
  in $\mathbb{C}^n$}, J. of Geom. Anal. \textbf{9} (1999), 93--117.

\bibitem[For04]{for04}
J.~E. Forn{\ae}ss, \emph{Short {$\mathbb{C}^k$}}, Complex analysis in several
  variables---{M}emorial {C}onference of {K}iyoshi {O}ka's {C}entennial
  {B}irthday, Adv. Stud. Pure Math., vol.~42, Math. Soc. Japan, 2004,
  pp.~95--108.

\bibitem[FS04]{fs04}
J.~E. Forn{\ae}ss and B.~Stens{\o}nes, \emph{Stable manifolds of holomorphic
  hyperbolic maps}, Internat. J. Math. \textbf{15} (2004), 749--758.

\bibitem[JV02]{jv02}
M.~Jonsson and D.~Varolin, \emph{Stable manifolds of holomorphic
  diffeomorphisms}, Invent. Math. \textbf{149} (2002), 409--430.

\bibitem[Ma{\~{n}}83]{man83}
R.~Ma{\~{n}}\'e, \emph{Ergodic theory and differentiable dynamics}, Springer,
  1983.

\bibitem[Pet05]{pet05}
H.~Peters, \emph{Non-autonomous complex dynamical systems}, Ph.D. thesis,
  University of Michigan, 2005.

\bibitem[Pet07]{pet07}
H.~Peters, \emph{Perturbed basins of attraction}, Math. Ann. \textbf{337} (2007),
  1--13.

\bibitem[PVW08]{pvw08}
H.~Peters, L.~R. Vivas, and E.~F. Wold, \emph{Attracting basins of volume
  preserving automorphisms of $\mathbb{C}^k$}, Internat. J. Math. \textbf{19}
  (2008), 801--810.

\bibitem[PW05]{pw05}
H.~Peters and E.~F. Wold, \emph{Non-autonomous basins of attractions and their
  boundaries}, J. Geom. Anal. \textbf{15} (2005), 123--136.

\bibitem[RR88]{rr88}
J.-P. Rosay and W.~Rudin, \emph{Holomorphic maps from {$\mathbb{C}^ n$} to
  {$\mathbb{C}^n$}}, Trans. Amer. Math. Soc. \textbf{310} (1988), 47--86.

\bibitem[Ste57]{ste57}
S.~Sternberg, \emph{Local contractions and a theorem of {P}oincar\'e}, Amer. J.
  Math. \textbf{79} (1957), 809--824.

\bibitem[Wei97]{wei97}
B.~Weickert, \emph{Automorphisms of $\mathbb{C}^n$}, Ph.D. thesis, University
  of Michigan, 1997.

\bibitem[Wol05]{wol05}
E.~F. Wold, \emph{Fatou-{B}ieberbach domains}, Internat. J. Math. \textbf{16}
  (2005), 1119--1130.

\end{thebibliography}

\providecommand{\bysame}{\leavevmode\hbox to3em{\hrulefill}\thinspace}
\providecommand{\MR}{\relax\ifhmode\unskip\space\fi MR }
\providecommand{\MRhref}[2]{%
  \href{http://www.ams.org/mathscinet-getitem?mr=#1}{#2}
}
\providecommand{\href}[2]{#2}

{\small\obeylines
\noindent Marco Abate, Alberto Abbondandolo, Pietro Majer
\noindent Dipartimento di Matematica
\noindent Universit\`a di Pisa
\noindent Largo Pontecorvo 5
\noindent 56127 Pisa
\noindent Italy
\noindent \emph{E-mails:} abate@dm.unipi.it
\noindent abbondandolo@dm.unipi.it
\noindent majer@dm.unipi.it}

\end{document}